\documentclass[11pt]{amsart}

\usepackage{amssymb}
\usepackage{eucal}
\usepackage{color}

\DeclareSymbolFont{SY}{U}{psy}{m}{n}
\DeclareMathSymbol{\emptyset}{\mathord}{SY}{'306}

\theoremstyle{plain}

\newtheorem{thm}{Theorem}[section]
\newtheorem{cor}[thm]{Corollary}
\newtheorem{lem}[thm]{Lemma}
\newtheorem{prop}[thm]{Proposition}
\theoremstyle{definition}
\newtheorem{defn}[thm]{Definition}
\newtheorem{rem}[thm]{Remark}

\numberwithin{equation}{section}

\def\g{\gamma}

\def\wi{\widetilde}

\def\beq{\begin{eqnarray}}
\def\eeq{\end{eqnarray}}
\def\beqa{\begin{eqnarray*}}
\def\eeqa{\end{eqnarray*}}

\begin{document}
\title{Rigidity of the flag structure for a class of  Cowen-Douglas operators}
\author[Ji]{Kui Ji}
\address[Ji]{Department of Mathematics, Hebei Normal University,
Shijiazhuang, Heibei 050016, China} \email{jikuiji@gmail.com}

\author[Jiang]{Chunlan Jiang}
\address[Jiang]{Department of Mathematics, Hebei Normal University,
Shijiazhuang, Heibei 050016, China} \email{cljiang@gmail.com}

\author[Keshari]{Dinesh Kumar Keshari}
\address[Keshari]{Department of Mathematics, Texas A\&M University,
College Station, TX 77843, USA} \email{kesharideepak@gmail.com}
\author[Misra]{Gadadhar Misra}
\address[Misra]{Department of Mathematics, Indian Institute of Science,
Bangalore 560012, India} \email{gm@math.iisc.ernet.in}
\thanks{The work of K. Ji was supported by the Foundation for
the Author of National Excellence Doctoral Dissertation of China
(Grant No. 201116). The work of C. Jiang was supported by
 National Natural Science Foundation of China (Grant No. A010602). The work of D.K.
Keshari was supported by a Research Associate fellowship of Indian
Institute of Science. The work of G. Misra was supported by J C
Bose National Fellowship and the UGC, SAP-IV}
\begin{abstract}
The explicit description of irreducible homogeneous operators in
the Cowen-Douglas class and the localization of Hilbert modules naturally leads to the definition of a smaller class of Cowen-Douglas operators possessing a flag structure. These operators are shown to be irreducible. It is also shown that the flag structure is rigid in that the unitary equivalence class of the operator and the flag structure determine each other. We obtain a complete set of unitary invariants which are somewhat more tractable than those of an arbitrary operator in the Cowen-Douglas class.
\end{abstract}

\maketitle
\section{Introduction}
Let $\mathcal H$ be a complex separable Hilbert space and
$\mathcal L(\mathcal H)$ denote the collection of bounded linear
operators on $\mathcal H$. The following important class of
operators was introduced in \cite{cd}.

\begin{defn}
For a connected open subset $\Omega$ of $\mathbb C$ and a positive
integer $n$, let
\begin{eqnarray*}
B_n(\Omega)  =  &\big \{& T\in\mathcal L(\mathcal H)\,|\,\,
\Omega\subset\sigma(T),\\
&& {\mathrm{ran}}\,(T-w)= \mathcal H\mbox{ for }w\in\Omega, \\
&&\bigvee_{w\in\Omega}\ker(T-w)= \mathcal H,\\
&&\dim~\ker(T-w)= n\mbox{ for } w\in\Omega\,\,  \big \},
\end{eqnarray*}
where $\sigma(T)$ denotes the spectrum of the operator $T$.
\end{defn}

We recall (cf. \cite{cd}) that an operator $T$ in the class
$B_n(\Omega)$ defines a holomorphic Hermitian vector bundle $E_T$
in a natural manner. It is the holomorphic sub-bundle of the trivial bundle $\Omega\times\mathcal H$ defined by
$$E_T= \{(w, x)\in\Omega\times\mathcal H: x\in \ker(T-w)\}$$
with the natural projection map $\pi:E_T\to \Omega$,  $\pi(w, x)=
w$. It is shown in \cite[Proposition 1.12]{cd} that if  $T$ is in $B_n(\Omega),$ then the mapping
$w\mapsto\ker(T-w)$ defines a rank $n$  holomorphic Hermitian
vector bundle $E_T$ over $\Omega.$  We reproduce below one of the main results from \cite{cd}. 
\begin{thm}\label{eq}
The operators $T$ and $\wi T$ in $B_n(\Omega)$ are unitarily
equivalent if and only if the corresponding holomorphic Hermitian
vector bundles $E_T$ and $E_{\wi T}$ are equivalent.
\end{thm}
They also find a set of complete invariant for this equivalence
consisting of curvature of $E_T$ and its covariant derivatives.
Unfortunately, these invariants are not easy to compute except
when the rank of the bundle is $1$. In this case, the curvature
\begin{eqnarray*}\mathcal K(w)\, dw\wedge\,d\bar{w} = -
\frac{~\partial^2\log\parallel{\gamma(w)}\parallel^2}{\partial{w}
\partial{\overline{w}}}dw\wedge\,d\bar{w}
\end{eqnarray*}
of the line bundle $E_T$, defined with respect to a non-zero
holomorphic section $\g$ of $E_T$, is a complete unitary invariant
of the operator $T.$ The
definition of the curvature, in this case, is independent of the choice of the non-vanishing section $\g$:  If $\g_0$ is another holomorphic (non-vanishing) section of $E$, then $\g_0=\phi\g$ for some holomorphic function $\phi$ on an open subset $\Omega_0$ of $\Omega$, consequently the harmonicity of log$|\phi|$ completes the verification. However, if the rank of the vector bundle is strictly greater than $1,$ then only the eigenvalues of the curvature are independent of the choice of the holomorphic frame. This limits the use of the curvature and its covariant derivative if the rank of the bundle is not $1.$ 
It is difficult to determine, in general, when an operator
$T\in B_n(\Omega)$ is irreducible, again except in the case $n=1$.
In this case, the rank of the vector bundle is $1$ and therefore it
is irreducible and so is the operator $T$. 

In this paper, we isolate a subset of irreducible operators in the Cowen-Douglas class $B_n(\Omega)$ for which a complete set of tractable unitary invariants is relatively easy to identify. We discuss this new class of operators in $B_2(\Omega)$ separately and then provide the details for the case of $n > 2.$ One important reason for separating out the case of $n=2$ is that the proofs that appear in this case are often necessary to begin an inductive proof in the case of an arbitrary $n\in \mathbb N.$ 

In a forthcoming paper, we construct similarity invariants for the operators in this new class. A generalization to the case of commuting tuples of operators is apparent which we intend to consider in future work. 

The results of this paper were announced in \cite{jjkm} and was the topic of a talk presented by the last author in the Workshop ``Hilbert Modules and Complex Geometry'' held during Apr 20 - 26, 2014 at Oberwolfach.

\subsection*{\large \sf Acknowledgement}The research reported here was started during a post doctoral visit of  Kui Ji to the Indian Institute of Science and was completed during a month long research visit of G. Misra to Hebei Normal university. We thank both of these Institutions for their admirable hospitality.
\section{A new class of operators in $B_2(\Omega)$}
\subsection{\large \sf Definitions} If  $T$ is an operator in ${B}_2(\Omega),$ then there exists a pair of operators  $T_0$ and $T_1$ in ${B}_1(\Omega)$ and a bounded operator $S$ such that  
$T=\Big ( \begin{smallmatrix}
T_0 & S \\
0 & T_1 \\
\end{smallmatrix}\Big ).$ This is Theorem 1.49 of \cite[page 48]{jw}. We show, the other way round, that two operators $T_0$ and $T_1$ from $B_1(\Omega)$ combine with the aid of an arbitrary bounded linear operator $S$ to produce an operator in $B_2(\Omega).$ 
\begin{prop}\label{prop2}
Let $T$ be a bounded linear operator of the form  $\Big ( \begin{smallmatrix}
T_0 & S \\
0 & T_1 \\
\end{smallmatrix}\Big ).$ Suppose that the two operators $T_0, T_1$ are in $B_1(\Omega).$ Then the operator $T$ is in $B_2(\Omega).$
\end{prop}
\begin{proof} 
Suppose $T_0$ and $T_1$ are defined on the Hilbert spaces $\mathcal H_0$ and $\mathcal H_1,$ respectively. Elementary considerations from index theory of Fredholm operators shows that the operator $T$ is Fredholm and $\mbox{\rm ind}(T) = \mbox{\rm ind}(T_0) + \mbox{\rm ind}(T_1)$ (cf. \cite[page 360]{Con}). Therefore, to complete the proof that $T$ is in $B_2(\Omega),$ all we have to do is prove that the vectors in the kernel $\ker (T -w),$ $w\in \Omega,$ span the Hilbert space $\mathcal H = \mathcal H_0 \oplus \mathcal H_1.$

Let $\gamma_0$ and $t_1$ be non-vanishing holomorphic sections for the two line bundles $E_0$ and $E_1$ corresponding to the operators $T_0$ and $T_1,$ respectively. 
For each $w\in \Omega,$ the operator $T_0-w_0$ is surjective. Therefore we can find a vector $\alpha(w)$ in $\mathcal H_0$ such that $(T_0-w) \alpha(w) = - S(t_1(w)),$ $w\in \Omega.$  Setting $\gamma_1(w) = \alpha(w) + t_1(w),$ we see that 
$$
(T-w) \gamma_1(w) = 0 = (T-w) \gamma_0(w).
$$
Thus  $\{\gamma_0(w),\gamma_1(w)\}\subseteq \ker\,(T-w)$ for $w$ in   $\Omega.$ If $x$ is any vector orthogonal to $\ker (T-w),$ $w\in \Omega,$ then in particular it is orthogonal to the vectors $\gamma_0(w)$ and $\gamma_1(w),$ $w\in \Omega,$ forcing it to be the zero vector.
\end{proof}

We impose one additional condition on these operators, namely,  $T_0 S = ST_1$ and assume that the operator $S$ is non-zero.   
With this seemingly innocuous hypothesis, we show that 
\emph{(i)} it is irreducible,  \emph{(ii)} and that any intertwining unitary operator between two of these operators must be diagonal and \emph{(iii)} the curvature of $E_{T_0}$ together with the second fundamental form of the inclusion $E_{T_0}\subseteq E_T$ form a complete set of unitary invariants for the operator $T.$  It is therefore natural to isolate this class of operators.  


\begin{defn}
We let $\mathcal{F}B_2(\Omega)$ denote the set of all bounded linear operators $T$  of the form
$T=\Big ( \begin{smallmatrix}
 T_0 & S \\
 0 & T_1 \\
\end{smallmatrix}\Big ),$ 
where the two operators $T_0,T_1$ are assumed to be in the Cowen-Douglas class $B_1(\Omega)$ and the operator $S$ is assumed to be a non-zero intertwiner  between them, that is, $T_0S=T_1S.$

Specifically, if the operator $T_i,$ $i=0,1,$ is defined on the  separable complex Hilbert space $\mathcal H_i,$ then $S$ is assumed to be a non-zero bounded linear operator from $\mathcal H_1$ to $\mathcal H_0$ such that $T_0S=T_1S.$ The operator $T$ is defined on the Hilbert space $\mathcal H:= \mathcal H_0\oplus \mathcal H_1.$
\end{defn}
Each of the operators in $\mathcal FB_2(\Omega)$ is also in the Cowen-Douglas class $B_2(\Omega)$ by virtue of Proposition \ref{prop2}. Thus $\mathcal FB_2(\Omega) \subseteq B_2(\Omega).$

Although, in the definition of the class $\mathcal FB_2(\Omega)$ given above, we have only assumed that $S$ is non-zero, its range must be dense as is shown below. 
\begin{prop}\label{dr}
Suppose $T_0$ and $T_1$ are two operators in $B_1(\Omega)$, and $S$ is a bounded operator intertwining $T_0$ and $T_1,$ that is, $T_0 S= S T_1$. Then $S$ is non zero if and only if range of $S$ dense if and only if $S^*$ is
injective.
\end{prop}
\begin{proof}
Let $\gamma$ be a holomorphic frame of $E_{T_1}$. Assume that $S$ is a non zero operator. The intertwining relationship $T_0 S=S T_1$ implies that $S\circ\gamma$ is a section of $E_{T_0}$. Clearly, there exists an open set $\Omega_0$ contained in $\Omega$ such that $S\circ \gamma$ is not zero on $\Omega_0,$ otherwise $S$ has to be zero. Since $S(\gamma)$ is a holomorphic frame of $E_{T_0}$ on $\Omega_0,$  it follows that the closure of the linear span of the vectors  
$\{S(\gamma(w)):w\in\Omega_0\}$ must equal 
$\mathcal{H}_0.$ Hence the range of the operator $S$ is dense.
\end{proof}
The following Proposition provides several equivalent characterizations of operators in the class $\mathcal FB_2(\Omega).$
\begin{prop}\label{f}
Suppose $T$ is a bounded linear operator on a Hilbert space
$\mathcal{H},$ which is in ${B}_2(\Omega).$ Then the following
conditions are equivalent.
\begin{enumerate}
\item[(i)] There exist an orthogonal decomposition
$\mathcal{H}_0\oplus\mathcal{H}_1$ of $\mathcal{H}$ and operators
$T_0:\mathcal{H}_0\to\mathcal{H}_0,$
$T_1:\mathcal{H}_1\to\mathcal{H}_1,$ 
and $S:\mathcal{H}_1\to\mathcal{H}_0$ such that $T=\begin{pmatrix}
                                                  T_0 & S \\
                                                  0 & T_1 \\
                                                \end{pmatrix},$ where  $T_0,T_1\in {B}_1(\Omega)$ and $T_0S=ST_1$, that is, $T\in\mathcal{F}B_2(\Omega)$.
\item[(ii)] There exists a holomorphic frame
$\{\gamma_0,\gamma_1\}$ of $E_{T}$ such that
$\tfrac{\partial}{\partial
w}\|\gamma_0(w)\|^2=\langle\gamma_1(w),\gamma_0(w)\rangle$.
\item[(iii)] There exists a holomorphic frame
$\{\gamma_0,\gamma_1\}$ of $E_{T}$ such that $\gamma_0(w)$ and
$\tfrac{\partial}{\partial w}\gamma_0(w)-\gamma_1(w)$ are
orthogonal for all $w$ belong to $\Omega$.
\end{enumerate}
\end{prop}
\begin{proof}
$(i)\Longrightarrow(ii)\!\!:$ Pick any two non-vanishing holomorphic sections $t_0$ and $t_1$ for the line bundles $E_{T_0}$ and $E_{T_1}$ respectively. Then 
\begin{eqnarray*}(T-w)t_1(w)&=&
(T_1-w)t_1(w)+S(t_1(w))\\
&=& S(t_1(w))\\
&=& \psi(w)t_0(w)
\end{eqnarray*}
for some holomorphic function $\psi$ defined on $\Omega.$ Setting 
$\gamma_0(w):=\psi(w)t_0(w)$ and
$\gamma_1(w):=\tfrac{\partial}{\partial w}\gamma_0(w)-t_1(w)$, we
see that 
$\{\gamma_0(w),\gamma_1(w)\}\subset\ker\,(T-w)$. Now assume that 
\begin{eqnarray}\label{e6}\alpha_0\gamma_0(w)+\alpha_1\gamma_1(w)=0\end{eqnarray}
for a pair of complex numbers $\alpha_0$ and $\alpha_1.$ Then 
\begin{eqnarray}
0&=&\langle\alpha_0\gamma_0(w)+\alpha_1\gamma_1(w),t_1(w)\rangle\nonumber\\
&=&\alpha_1\langle\gamma_1(w),t_1(w)\rangle\nonumber\\
&=&-\alpha_1\|t_1(w)\|^2\label{e7}.
\end{eqnarray}
From equations (\ref{e6}) and (\ref{e7}), it follows that
$\alpha_0=\alpha_1=0$. Thus $\{\gamma_0,\gamma_1\}$ is a
holomorphic frame of $E_T$.
Since $\langle t_1(w),\gamma_0(w)\rangle=0,$ 
we see that 
$$ \tfrac{\partial}{\partial w}\|\gamma_0(w)\|^2= \langle\gamma_1(w),\gamma_0(w)\rangle.$$

$(ii)\Longleftrightarrow(iii)\!:$ This equivalence is evident from the definition.

$(iii)\Longrightarrow (i)\!\!:$  Set $t_1(w):=\tfrac{\partial}{\partial
w}\gamma_0(w)-\gamma_1(w).$  Let $\mathcal{H}_0$ and $\mathcal{H}_1$ be the closed linear span of $\{\gamma_0(w):w\in\Omega\}$ and
$\{t_1(w):w\in\Omega\},$ respectively.
Set $T_0=T_{|\mathcal{H}_0}$, $T_1=P_{\mathcal{H}_1}T_{|\mathcal{H}_1}$ and $S=P_{\mathcal{H}_0}T_{|\mathcal{H}_1}.$

We see that the closed linear span of the vectors 
$\{\gamma_0(w), t_1(w): w\in\Omega\}$ is $\mathcal{H}:$
Suppose $x$ in $\mathcal{H}$ is orthogonal to this set of vectors. Then clearly,  $x\bot \gamma_0(w)$ and $x\bot
t_1(w)$ for all $w$ in $\Omega.$ Or, equivalently 
$x\bot\gamma_0(w)$ and $x\bot\gamma_1(w)$ for all $w$ 
in $\Omega.$ Therefore $x$ must be the $0$ vector.
Next, we show that the two operators $T_0$ and $T_1$ are in 
$B_1(\Omega).$

Clearly, $(T_1-w)$ is onto. Thus $\mbox{index}\,(T_1-w)=\dim~\ker\,(T_1-w)$
and
$2=\mbox{index}\,(T-w)=\mbox{index}\,(T_0-w)+\mbox{index}\,(T_1-w).$
It follows that $\dim~\ker(T_1-w)=1\;\mbox{or}\;2 $.  

Suppose $\dim~\ker\,(T_1-w)=2$ and $\{s_1(w),s_2(w)\}$ be a holomorphic choice of linearly independent vectors in $\ker\,(T_1-w)$. Then we can find holomorphic functions
$\phi_1,\phi_2$ defined on $\Omega$ such that $S(s_1(w))=\phi_1(w)\gamma_0(w)$ and
$S(s_2(w))=\phi_2(w)\gamma_0(w).$ Setting 
$\tilde{\gamma}_0(w):=\gamma_0(w),\tilde{\gamma}_1(w):=\frac{\partial}{\partial
w}(\phi_1(w)\gamma_0(w))-s_1(w)$ and
$\tilde{\gamma}_2(w):=\frac{\partial}{\partial
w}(\phi_2(w)\gamma_0(w))-s_2(w),$ we see that 
$(T-w)(\tilde{\gamma}_i(w))=0$ for $0\leq i\leq 2.$ If $\sum_{i=0}^2\alpha_i \tilde{\gamma}_i(w)=0,$ $\alpha_i \in\mathbb C,$ then  \begin{eqnarray*}
\alpha_0 \gamma_0(w)+\tfrac{\partial}{\partial
w}\big((\alpha_1
\phi_1(w)+\alpha_2\phi_2(w))\gamma_0(w)\big)+\alpha_1
s_1(w)+\alpha_2 s_2(w) = 0.
\end{eqnarray*}
It follows that $\alpha_1 s_1(w)+\alpha_2 s_2(w)=0$ since $\mathcal{H}_0$ is orthogonal to $\mathcal{H}_1.$ Hence 
$\alpha_1=\alpha_2=0$ implying $\alpha_0=0.$ Thus we have
$\dim~\ker(T-w)\geq 3.$ This contradiction proves that $\dim \ker (T_0-w) =1$ and hence $T_1$ is in $B_1(\Omega).$

To show that $T_0$ is in $B_1(\Omega),$ pick any $x\in\mathcal{H}_0$, and note that there exist $z\in\mathcal{H}$ such that $(T-w)z=x$ since $T-w$ is onto. Let $z_{\mathcal{H}_1}$ and $z_{\mathcal{H}_0}$ be the projections of $z$ to the subspaces $\mathcal H_0$ and $\mathcal H_1,$ respectively. We have 
$[(T_0-w)z_{\mathcal{H}_0}+S(z_{\mathcal{H}_1})]+(T_1-w)z_{\mathcal{H}_1}=x.$
Therefore $(T_1-w)z_{\mathcal{H}_1}=0$ and
$(T_0-w)z_{\mathcal{H}_0}+S(z_{\mathcal{H}_1})=x$. Since
$\mbox{dim}\,\ker\,(T_1-w)=1\,$, so $z_{\mathcal{H}_1}=c_1t_1(w),$
it follows that 
\begin{eqnarray*}x&=&(T_0-w)z_{\mathcal{H}_0}+S(z_{\mathcal{H}_1})\\
&=&(T_0-w)z_{\mathcal{H}_0}+S(c_1t_1(w))\\
&=&(T_0-w)z_{\mathcal{H}_0}+c_1\gamma_0(w)\\
&=&(T_0-w)z_{\mathcal{H}_0}+(T_0-w)(c_1\tfrac{\partial}{\partial w}\gamma_0(w))\\
&=&((T_0-w)(z_{\mathcal{H}_0}+c_1\tfrac{\partial}{\partial
w}\gamma_0(w)).
\end{eqnarray*}
Thus $T_0-w$ is onto. We have $2=\dim~\ker\;(T-w)=
\dim~\ker\;(T_0-w)+\dim~\ker\;(T_1-w).$  Hence $\dim~\ker\;(T_0-w)=1$ and we see that $T_0$ is in $B_1(\Omega).$

Finally, since $S(t_1(w))= \gamma_0(w),$  it
follows that $T_0S=ST_1$.
\end{proof}

\subsection{\large \sf Models for operators in $\mathcal FB_2(\Omega)$} An operator $T\in \mathcal{F}B_2(\Omega)$ is also in $B_2(\Omega),$ therefore as is well-known (cf. \cite{cd, CS}), it can be realized as the adjoint of a multiplication operator on some reproducing kernel
Hilbert space of holomorphic $\mathbb{C}^2$-valued functions. These functions are defined on $\Omega^*:=\{w:\bar{w} \in \Omega\}.$ 
An explicit description for operators in $\mathcal FB_2(\Omega)$ follows.

Let $E_{T}$ be the holomorphic Hermitian  vector
bundle over $\Omega$ corresponding to the operator $T.$   Since $T$ is in $\mathcal FB_2(\Omega),$ we may find a holomorphic frame $\gamma=\{\gamma_0,\gamma_1\}$ such that  $\gamma_0(w)$ and $\tfrac{\partial}{\partial w}\gamma_0(w)-\gamma_1(w)$ are orthogonal for all $w$ in $\Omega.$  Define  $\Gamma:\mathcal{H}\to
\mathcal{O}(\Omega^*,\mathbb{C}^2)$ as follows:
$$ \Gamma(x)(z)=\big( \langle x,\gamma_0(\bar z)\rangle ,
\langle x,\gamma_{1}(\bar z)\rangle\big)^{\rm tr}\;\;\;\;\;\;\;
z\in\Omega^*,\;x\in\mathcal{H},
$$
where $\mathcal{O}(\Omega^*,\mathbb{C}^2)$ is the space of
holomorphic functions defined on $\Omega^*$ which take values in
$\mathbb{C}^2$. Here $(\,\,\cdot\,\, , \,\, \cdot \,\,)^{\rm tr}$ denotes the transpose of the vector $(\,\,\cdot\,\,, \,\,\cdot\,\,).$ 

The map $\Gamma$ is injective and therefore transplanting the inner product from
$\mathcal{H}$ on the range of $\Gamma,$ we make it unitary
from $\mathcal{H}$ onto $\mathcal{H}_{\Gamma}:=\mbox{ran}\, \Gamma.$ Define
$K_{\Gamma}$ to be the function on $\Omega^*\times \Omega^*$ 
taking values in the $2\times 2$ matrices
$\mathcal{M}_2(\mathbb{C}):$
\renewcommand\arraystretch{1.275}
\begin{eqnarray}\label{canonicalKGamma}
K_{\Gamma}(z,w)&=& \big(\!\big(\langle \gamma_j(\bar
w),\gamma_i(\bar z)\rangle\big)\!\big)_{i,j=0}^{1}\nonumber \\
&=& \begin{pmatrix}
  \langle \gamma_0(\bar w),\gamma_0(\bar z)\rangle &
  \frac{\partial}{\partial \bar w}\langle \gamma_0(\bar w),\gamma_0(\bar z)\rangle \nonumber\\
  \frac{\partial}{\partial z} \langle \gamma_0(\bar w),\gamma_0(\bar z)\rangle
  &
 \frac{\partial^2}{\partial z\partial \bar w} \langle \gamma_0(\bar w),
 \gamma_0(\bar z)\rangle+\langle t_1(\bar w),t_1(\bar z)\rangle \\
\end{pmatrix}\\
&=&
\begin{pmatrix}
  K_0(z,w) & \frac{\partial}{\partial \bar w}K_0(z,w) \\
  \frac{\partial}{\partial z}K_0(z,w) & \frac{\partial^2}{\partial z\partial \bar w}K_0(z,w) \\
\end{pmatrix}+ \begin{pmatrix}
 0 & 0 \\
 0 & K_1(z,w) \\
\end{pmatrix},
\end{eqnarray}
\renewcommand\arraystretch{1}
where $t_1(\bar w)=\tfrac{\partial}{\partial \bar w}\gamma_0(\bar
w)-\gamma_1(\bar w)$, $K_0(z,w)=\langle \gamma_0(\bar
w),\gamma_0(\bar z)\rangle$ and $K_1(z,w)= \langle t_1(\bar
w),t_1(\bar z)\rangle$ for $z,w\in\Omega^*$. Set
$(K_{\Gamma})_w(\cdot)=K_{\Gamma}(\cdot,w)$. It is then easily
verified that $K_\Gamma$ has the following properties:
\begin{enumerate}
\item The reproducing property: $\langle\Gamma(x)(\cdot),(K_{\Gamma})_w(\cdot)\eta\rangle_{\rm{ran}\:\Gamma} = \langle
\Gamma(x)(w),\eta\rangle_{\mathbb{C}^2},\,x\in\mathcal{H},$ $\eta\in\mathbb{C}^2,$ $w\in\Omega^*.$
\item The unitary operator $\Gamma$ intertwines the operators $T$ defined on $\mathcal H$ and $M^*$ defined on $\mathcal H_\Gamma,$ namely, $\Gamma T^*= M_z \Gamma.$
\item Each $w$ in $\Omega$ is an eigenvalue with eigenvector $(K_{\Gamma})_{\bar{w}}(\cdot)\eta,$ $\eta \in \mathbb C^2,$ for the operator $M^*=\Gamma T\Gamma^*.$ 
\end{enumerate}

\subsection{\large \sf Rigidity} Once we represent an operator $T$ from $\mathcal FB_2(\Omega)$ in this form, the  possibilities for the change of frame are limited. The admissible ones are described in the following lemma.  
\begin{lem} \label{lut} Let $T$ be an operator in $\mathcal{F}B_2(\Omega).$ Suppose  $\{\gamma_0,\gamma_1\}$, $\{\tilde{\gamma}_0,\tilde{\gamma}_1\}$ are two
frames of the vector bundle $E_T$ such that  $\gamma_0(w)\bot
(\tfrac{\partial}{\partial w}\gamma_0(w)-\gamma_1(w))$ and
$\tilde{\gamma}_0(w) \bot (\tfrac{\partial}{\partial
w}\tilde{\gamma}_0(w)-\tilde{\gamma}_1(w))$ for all $w\in \Omega.$
If  $\phi=\begin{pmatrix}
          \phi_{11} & \phi_{12} \\
         \phi_{21} & \phi_{22} \\
        \end{pmatrix}$ is any change of frame between $\{\g_0,\g_1\}$ and $\{\tilde{\g}_0,\tilde{\g}_1\}$, that
        is,
$$\{\tilde{\g}_0,\tilde{\g}_1\}=\{\g_0,\g_1\}\begin{pmatrix}
          \phi_{11} & \phi_{12} \\
         \phi_{21} & \phi_{22} \\
        \end{pmatrix},$$
then $\phi_{21}=0,\,\phi_{11}=\phi_{22}$ and
$\phi_{12}=\phi_{11}^{\prime}.$
\end{lem}

\begin{proof}

Define the unitary map $\Gamma,$ as above, using the holomorphic frame $\gamma=\{\gamma_0,\gamma_1\}.$ The operator $T$ is then unitarily equivalent to the adjoint of the multiplication operator on the Hilbert space $\mathcal H_\Gamma$ possessing a reproducing kernel $K_\Gamma$ of the form \eqref{canonicalKGamma}.  
%
Let $e_1$ and $e_2$ be the standard unit vectors in $\mathbb C^2.$  Clearly, ${(K_{\Gamma})}_w(\cdot)e_1$ and ${(K_{\Gamma})}_w(\cdot)e_2$ are two linearly independent eigenvectors of $M^*$ with eigenvalue $\bar{w}.$
%
%
%

Similarly, corresponding to the holomorphic frame $\tilde{\gamma}=\{\tilde{\gamma_0},\tilde{\gamma_1}\},$ the operator $T$ is unitarily equivalent to the 
adjoint of multiplication operator on the Hilbert space $\mathcal{H}_{\tilde{\Gamma}}.$ The reproducing kernel $K_{\tilde{\Gamma}}$ is again of the form \eqref{canonicalKGamma} except that $K_0$ and $K_1$ must be replaced by $\tilde{K}_0$ and $\tilde{K}_1,$ respectively.  

For $i=0,1,$ set $s_i(w) := (K_\Gamma)(w) e_i,$ and $\tilde{s}_i(w) := (K_{\tilde{\Gamma}})(w) e_i.$
%
%
%
%
%
%
%
%
Let $\phi(w):=
\begin{pmatrix}
  \phi_{00}(w) & \phi_{01}(w) \\
  \phi_{10}(w) & \phi_{11}(w) \\
\end{pmatrix}
$ be the holomorphic function, taking values in $2\times 2$ matrices,  such that
$$(\tilde{s}_0(w),\tilde{s}_1(w))=(s_0(w),s_1(w))\phi(w).$$
This implies that
\begin{eqnarray}\label{eut1}\tilde{s}_0(w)=\phi_{00}(w)s_0(w)+\phi_{10}(w)s_1(w)\end{eqnarray} and
\begin{eqnarray}\label{eut2}\tilde{s}_1(w)=\phi_{01}(w) s_0(w)+\phi_{11}(w)s_1(w).\end{eqnarray}
From Equation (\ref{eut1}), equating the first and the second coordinates separately, we have 
\begin{eqnarray}\label{eut3}(\tilde{K}_0)_w(\cdot)=\phi_{00}(w)
(K_0)_w(\cdot)+\phi_{10}(w)\tfrac{\partial}{\partial \bar w}
(K_0)_w(\cdot)\end{eqnarray} 
and
\begin{eqnarray}\label{eut4}\tfrac{\partial}{\partial z}(\tilde{K_0})_w(\cdot)=
\phi_{00}(w)\tfrac{\partial}{\partial
z}(K_0)_w(\cdot)+\phi_{10}(w)\tfrac{\partial^2}{\partial z\partial
\bar w} (K_0)_w(\cdot)+ \phi_{10}(w)(K_1)_w(\cdot).\end{eqnarray}
From these two equations, we get
$$\phi_{00}(w)\tfrac{\partial}{\partial
z}(K_0)_w(\cdot)+\phi_{10}(w)\tfrac{\partial^2}{\partial z\partial
\bar w} (K_0)_w(\cdot)=\phi_{00}(w)\tfrac{\partial}{\partial
z}(K_0)_w(\cdot)+\phi_{10}(w)\tfrac{\partial^2}{\partial z\partial
\bar w} (K_0)_w(\cdot)+ \phi_{10}(w)(K_1)_w(\cdot),$$ which
implies that $\phi_{10}=0.$
Finally, from Equation (\ref{eut2}), we have 
\begin{eqnarray}\label{eut5}
\tfrac{\partial}{\partial \bar
w}(\tilde{K}_0)_w(\cdot)=\phi_{01}(w)
(K_0)_w(\cdot)+\phi_{11}(w)\tfrac{\partial}{\partial \bar w}
(K_0)_w(\cdot)
\end{eqnarray}
The Equations (\ref{eut2}) and (\ref{eut5}) together give
$$\phi_{01}=\phi_{00}^{\prime}\;\;\mbox{and} \;\;\phi_{00}=\phi_{11}$$
completing the proof.
\end{proof}
A very important consequence of this Lemma is that the decomposition of the operators in the class $\mathcal FB_2(\Omega)$ is unique in the sense described in the following proposition. 
\begin{prop}\label{mainp} Let $T,\tilde{T}\in \mathcal{F}B_2(\Omega)$ be two operators of of the form $\Big (\begin{smallmatrix} T_0 & S \\
0 & T_1
\end{smallmatrix} \Big )$ and $\Big (\begin{smallmatrix} \tilde{T}_0 & \tilde{S} \\
0 & \tilde{T}_1
\end{smallmatrix} \Big )$  with respect to the decomposition  $\mathcal H = \mathcal H_0 \oplus \mathcal H_1$ and $\tilde{\mathcal H} = \tilde{\mathcal H}_0 \oplus \tilde{\mathcal H}_1,$ respectively. 
Let  $U=\Big (\begin{smallmatrix}U_{11} & U_{12} \\
U_{21} & U_{22} \\
\end{smallmatrix}\Big ):\mathcal H_0 \oplus \mathcal H_1 \to \tilde{\mathcal H}_0 \oplus \tilde{\mathcal H}_1$ be an unitary operator such that
$$\begin{pmatrix}U_{11} & U_{12} \\
U_{21} & U_{22} \\
\end{pmatrix}\begin{pmatrix}T_0 & S \\
0 & T_1 \\
\end{pmatrix}=\begin{pmatrix}\tilde{T}_0 & \tilde{S} \\
0 & \tilde{T}_1 \\
\end{pmatrix}\begin{pmatrix}U_{11} & U_{12} \\
U_{21} & U_{22} \\
\end{pmatrix},$$ then $U_{12}=U_{21}=0$.
\end{prop}
\begin{proof}
Let $\{\gamma_0,\gamma_1\}$ and
$\{\tilde{\gamma_0},\tilde{\gamma_1}\}$ be holomorphic frames of
$E_{T}$ and $E_{\tilde T}$ respectively with the property that
$\gamma_0\perp (\tfrac{\partial}{\partial w}\gamma_0-\gamma_1)$
and $\tilde{\gamma}_0\perp (\tfrac{\partial}{\partial
w}\tilde{\gamma}_0-\tilde{\gamma}_1)$. Set
$t_1:=(\tfrac{\partial}{\partial w}\gamma_0-\gamma_1)$ and
$\tilde{t}_1:= (\tfrac{\partial}{\partial
w}\tilde{\gamma}_0-\tilde{\gamma}_1)$. Since $U$ intertwines $T$
and $\tilde T$, it follows that $\{U\gamma_0,U\gamma_1\}$ is a second holomorphic frame of $E_{\tilde T}$ with the property
$U\gamma_0\perp (\tfrac{\partial}{\partial
w}(U\gamma_0)-U\gamma_1)= U(t_1)$.  By Lemma \ref{lut}, we have
that 
\begin{eqnarray}\label{etd1}U(\gamma_0)=\phi\tilde{\gamma_0}\end{eqnarray} and
\begin{eqnarray}\label{etd2}U(\gamma_1)= \phi^{\prime}\tilde{\gamma}_0+\phi
\tilde{\gamma}_1.\end{eqnarray}
From equations \eqref{etd1} and \eqref{etd2}, we get
\begin{eqnarray}\label{etd3}U(t_1)=\phi\,\tilde{t}_1.\end{eqnarray} From equations \eqref{etd1}
and \eqref{etd3}, it follows that $U$ maps $\mathcal{H}_0$ to
$\mathcal{H}_0$  and $\mathcal{H}_1$ to $\mathcal{H}_1.$ Thus
$U$ is a block diagonal from $\mathcal{H}_0\oplus\mathcal{H}_1$ onto $\tilde{\mathcal{H}}_0\oplus\tilde{\mathcal{H}}_1$
\end{proof}
\begin{rem}\label{frame}
In summary, we note that 
a holomorphic change of frame for the vector bundle  $E_T,$ 
preserving the orthogonality relation between $\gamma_0$ and $\tfrac{\partial}{\partial w}\gamma_0(w)-\gamma_1(w),$ must be of the form $\Big (\begin{smallmatrix}\varphi & \varphi^\prime\\ 0& \varphi\end{smallmatrix}\Big ).$ 
Thus such a change of frame for  the vector bundle $E_T$ induces change of frame $\Big (\begin{smallmatrix}\varphi&0\\0&\varphi\end{smallmatrix}\Big )$ 
for the vector bundle $E_{\big (\begin{smallmatrix}T_0 & 0\\ 0&T_1\end{smallmatrix}\big )}$ and vice-versa. 
\end{rem}
\begin{cor}For $i=0,1,$ let $T_i$ be any two operators in ${B}_1(\Omega)$. Let $S$ and
$\tilde{S}$ be bounded linear operators such that $T_0 S=S T_1$ and $T_0\tilde{S}=\tilde{S}T_1.$ If
 $T=\Big (\begin{smallmatrix}
 T_0 & S \\
 0 & T_1 \\
 \end{smallmatrix} \Big )$ and  $\tilde T=\Big (\begin{smallmatrix}
 T_0 & \tilde{S} \\
 0 & T_1 \\
 \end{smallmatrix} \Big ),$ then $T$ is unitarily equivalent to $\tilde T$ if and only if $\tilde S= e^{i\theta}S$ for some real number $\theta.$ 
\end{cor}
\begin{proof}
Suppose that  $UT=\tilde{T} U$ for some unitary operator $U.$  We have just shown that such an operator $U$ must be diagonal, say $U=\Big (\begin{smallmatrix}
 U_{11}& 0 \\
 0 & U_{22} \\
 \end{smallmatrix}\Big ).$ Hence we have  
 \begin{eqnarray}\label{ecmp1}
 U_{11}T_0=T_0U_{11},\,\, U_{22}T_1=T_1U_{22},\,\, U_{11}S=\tilde{S}U_{22}.
 \end{eqnarray}
Since $U_{11}$ is unitary, the first of the equations (\ref{ecmp1}) implies that
$$U_{11}\in \{T_0,T_0^*\}^{\prime}:=\{W\in
\mathcal{L}(\mathcal{H}_0): WT_0=T_0W
\;\mbox{and}\;WT_0^*=T_0^*W\}.$$ Since $T_0$ is an irreducible
operator,  we conclude that $U_{11}=e^{i\theta_1}I_{\mathcal{H}_0}$ for
some $\theta_1\in\mathbb{R}.$ Similarly,
$U_{22}=e^{i\theta_2}I_{\mathcal{H}_1}$ for some
$\theta_2\in\mathbb{R}.$ Hence the third equation in (\ref{ecmp1}) implies that
$\tilde{S}=e^{i(\theta_1-\theta_2)}S$.

Conversely suppose that $\tilde{S}=e^{i\theta}S$ for some real
number $\theta$. Then evidently the operator $U:=\left ( \begin{smallmatrix}
 \exp\big (i\tfrac{\theta}{2}\big )I_0& 0 \\
 0 &  \exp\big (-i\tfrac{\theta}{2}\big )I_1 \\
 \end{smallmatrix} \right )$ is unitary on $\mathcal{H}=\mathcal{H}_0\oplus
 \mathcal{H}_1$ and $UT=\tilde{T}U.$
\end{proof}
\begin{cor}
For $i=0,1,$ let $T_i$ be two operators in ${B}_1(\Omega).$ Let $S$ be a non-zero bounded linear operators such that $T_0S=ST_1.$ If 
$T_{\mu}=\Big (\begin{smallmatrix}
 T_0 & \mu S\\
 0 & T_1 \\
 \end{smallmatrix}\Big )$ and $T_{\tilde{\mu}}=\big (\begin{smallmatrix}
 T_0 & \tilde{\mu} S\\
 0 & T_1 \\
 \end{smallmatrix}\Big ),$ $\mu,\,\tilde{\mu} > 0,$ then  $T_{\mu}$ is unitarily equivalent to $ T_{\tilde {\mu}}$ if and only if $\mu=\tilde {\mu}$.
\end{cor}
\subsection{\large \sf A complete set of unitary invariants} The following theorem lists a complete set of unitary invariants for operators in $\mathcal FB_2(\Omega).$
\begin{thm}\label{maint}
Suppose that $T=\Big (\begin{smallmatrix}T_0 & S \\
0 & T_1 \\
\end{smallmatrix}\Big )$
and $\tilde{T}=\Big (\begin{smallmatrix}\tilde{T}_0 & \tilde S \\
0 & \tilde{T}_1 \\
\end{smallmatrix}\Big )$ are any two operators in $\mathcal FB_2(\Omega).$ 
Then the operators $T$ and $\tilde{T}$ are unitarily equivalent if and only if $\mathcal{K}_{T_1}=\mathcal{K}_{\tilde{T}_1}$ (or,
$\mathcal{K}_{T_0}=\mathcal{K}_{\tilde{T}_0}$) and
$\frac{\|S(t_1)\|^2}{\|t_1\|^2}= \frac{\|\tilde
S(\tilde{t}_1)\|^2}{\|\tilde{t}_1\|^2},$ where $t_1$ and $\tilde{t}_1$ are non-vanishing holomorphic sections for the  vector bundles $E_{T_1}$ and $E_{\tilde{T}_1},$ respectively.
\end{thm}

\begin{proof}

By working on a sufficiently small open subset of $\Omega$, we can assume that $S(t_1)$
and $\tilde S(\tilde{t}_1)$ are holomorphic frames of the bundle
$E_{T_0}$ and $E_{\tilde{T}_0},$ respectively. First suppose that
$\bar{\partial}\partial\log\,\|S(t_1)\|^2=\bar{\partial}\partial\log\,\|\tilde
S(\tilde{t}_1)\|^2$ and $\frac{\|S(t_1)\|^2}{\|t_1\|^2}=
\frac{\|\tilde S(\tilde{t}_1)\|^2}{\|\tilde{t}_1\|^2}.$ 
Then we claim that $T$ and $\tilde{T}$ are unitarily equivalent.
The equality of the curvatures, namely, $\bar{\partial}\partial\log\,\|S(t_1)\|^2=\bar{\partial}\partial\log\,\|\tilde
S(\tilde{t}_1)\|^2$ implies that $\|S(t_1)\|^2=|\phi|^2\|\tilde
S(\tilde{t}_1)\|^2$ for some non-vanishing holomorphic function $\phi$ on $\Omega.$ It may be that we have to shrink, without loss of generality, to a smaller open set $\Omega_0.$ The second of our assumptions gives 
$\|t_1\|^2=|\phi|^2\|\tilde{t}_1\|^2.$ Let
$\gamma_0(w):=S(t_1(w))$ and $\tilde{\gamma}_0(w):=\tilde
S(\tilde{t}_1(w));$ $\gamma_1(w):=\frac{\partial}{\partial
w}\gamma_0(w)-t_1(w)$ and
$\tilde{\gamma}_1(w):=\frac{\partial}{\partial
w}\tilde{\gamma}_0(w)-\tilde{t}_1(w).$ It follows that $\{\gamma_0,
\gamma_1\}$ and $\{\tilde{\gamma}_0, \tilde{\gamma}_1\}$ are
holomorphic frames of $E_T$ and $E_{\tilde T},$ respectively.
Define the map $\Phi:E_T\to E_{\tilde{T}}$ as follows:
\begin{enumerate}
\item $\Phi(\gamma_0(w))=\phi(w)\tilde{\gamma}_0(w),$ 
\item $\Phi(\gamma_1(w))=\phi^{\prime}(w)\tilde{\g}_0(w)+\phi(w)\tilde{\gamma}_1(w).$
\end{enumerate}
Clearly, $\Phi$ is holomorphic. Note that \begin{eqnarray*}
\langle\Phi(\gamma_0(w)),\Phi(\gamma_1(w))\rangle&=&\langle\phi(w)\tilde{\gamma_0}(w),\phi^{\prime}(w)\tilde{\gamma_0}(w)+\phi(w)\tilde{\gamma}_1(w)\rangle\\
&=& \langle\phi(w)\tilde{\gamma_0}(w),\phi^{\prime}(w)\tilde{\gamma_0}(w)+\phi(w)(\tfrac{\partial}{\partial w}\tilde \gamma_0(w)-\tilde {t}_1(w))\rangle\\
&=&\langle\phi(w)\tilde{\gamma_0}(w),\tfrac{\partial}{\partial w}(\phi(w)\tilde{\gamma}_0(w))-\phi(w)\tilde{t}_1(w)\rangle\\
&=&\tfrac{\partial}{\partial \bar{w}}\|\phi(w)\tilde{\gamma}_0(w)\|^2\\
&=&\tfrac{\partial}{\partial \bar{w}}\|{\gamma}_0(w)\|^2
\end{eqnarray*}
and \begin{eqnarray*}
\langle\gamma_0(w),\gamma_1(w)\rangle &=& \langle\gamma_0(w),\tfrac{\partial}{\partial w}\gamma_0(w)-t_1(w)\rangle\\
&=&\tfrac{\partial}{\partial \bar{w}}\|{\gamma}_0(w)\|^2.
\end{eqnarray*}
Hence we have
$\langle\Phi(\gamma_0(w)),\Phi(\gamma_1(w))\rangle=\langle\gamma_0(w),\gamma_1(w)\rangle$.
Similarly, $\|\Phi(\gamma_0(w))\|=\|\gamma_0(w)\|$ and
$\|\Phi(\gamma_1)\|=\|\gamma_1\|$.
Thus $E_{T}$ and $E_{\tilde T}$ are equivalent 
as  holomorphic Hermitian vector bundles. Hence $T$ 
and $\tilde T$ are unitarily equivalent by Theorem \ref{eq} of Cowen and Douglas.

Conversely, suppose $T$ and $\tilde{T}$ are unitarily equivalent.
Let $U:\mathcal{H}\to \tilde{\mathcal{H}}$ be the unitary map such
that
$UT=\tilde{T}U$. By proposition \ref{mainp}, $U$ takes the form $\Big (\begin{smallmatrix}U_{1} & 0 \\
0& U_2 \\
\end{smallmatrix} \Big )$ for some pair of unitary operators $U_1$ and $U_2.$ Hence we have $U_1(S(t_1))= \phi_1(\tilde S(\tilde t_1))$ and $U_2 t_1=\phi_2 \tilde t_1$.
The intertwining relation $U_1S=\tilde S U_2$ implies that $\phi_1=\phi_2$. Thus
$\mathcal{K}_{T_0}=\mathcal{K}_{\tilde{T_0}}$ and
\begin{eqnarray*}
\frac{\|S(t_1)\|^2}{\|t_1\|^2}&=& \frac{\|U_1(S(t_1))\|^2}{\|U_2(t_1)\|^2}
= \frac{\|\phi_1\tilde S(\tilde t_1)\|^2}{\|\phi_2\tilde t_1\|^2}
= \frac{\|\tilde S(\tilde t_1)\|^2}{\|\tilde t_1\|^2}.
\end{eqnarray*}
This verification completes the proof.
\end{proof}
\subsection{\large\sf The second fundamental form}
We relate the invariants of Theorem \ref{maint} to the second fundamental form of the inclusion $E_0\subseteq E.$  The computation of the second fundamental form is given below  following \cite[page. 2244]{dm}. Here 
$E_0,$ is the line bundle corresponding to the operator $T_0$ and $E$ is the vector bundle of rank $2$ corresponding to the operator $T$ in $\mathcal FB_2(\Omega).$ 
Let  $\{\gamma_0,\gamma_1\}$ be a holomorphic frame for $E$ such that $\gamma_0$ and $t_1:=\partial\gamma_0-\gamma_1$ are orthogonal. One obtains an orthonormal frame, say, $\{e_0,e_1\}$, from the holomorphic frame $\{\gamma_0,\gamma_1\}$ by the usual Gram-Schmidt process -- Set $h = \langle\gamma_0,\gamma_0 \rangle,$ and observe that 
\begin{eqnarray*}e_1=h^{-1/2}\gamma_0,\,\, e_2=\frac{\gamma_1-\frac{\gamma_0\langle\gamma_1,\gamma_0\rangle}
{\|\gamma_0\|^2}}{(\|\gamma_1\|^2-\frac{|\langle\gamma_1,\gamma_0\rangle|^2}
{\|\gamma_0\|^2})^{1/2}}\end{eqnarray*}
are orthogonal. The canonical hermitian connection $D$ for the vector bundle $E_T$ is given, in terms of $e_1$ and $e_2$ by the formula: 
\begin{eqnarray*}
D\,e_1&=& D^{1,0}e_1+D^{0,1}e_1\\
&=&\alpha_{11} e_1+\alpha_{21}e_2+\bar{\partial} e_1\\
&=&(\alpha_{11}-\bar{\partial}(\log h))e_1+ \alpha_{21} e_2\\
&=&\theta_{11} e_1+\theta_{21}e_2,
\end{eqnarray*}
where $\alpha_{11},\alpha_{21}$ are $(1,0)$ forms to be determined. Similarly, we have
\begin{eqnarray*}
D\;e_2 &=& D^{1,0}e_2+D^{0,1}e_2\\
&=&\alpha_{12} e_1+\alpha_{22}e_2+\bar{\partial} e_2\\
&=&\left(\alpha_{12}- h^{1/2}\frac{\bar{\partial}(h^{-1}\langle\gamma_2,\gamma_1\rangle)}
{(\|\gamma_2\|^2-\frac{|\langle\gamma_2,\gamma_1\rangle|^2}
{\|\gamma_1\|^2})^{1/2}}\right)e_1+\left(\alpha_{22} -\frac{1}{2}\frac{\bar{\partial}(\|\gamma_2\|^2-\frac{\langle \gamma_2,\gamma_1\rangle}{\|\gamma_1\|^2})}
{(\|\gamma_2\|^2-\frac{\langle \gamma_2,\gamma_1\rangle}{\|\gamma_1\|^2})}\right)
 e_2\\
 &=&\theta_{12} e_1+\theta_{22}e_2,
\end{eqnarray*}
where $\alpha_{12},\alpha_{22}$ are $(1,0)$ forms to be determined. Since we are working with
an orthonormal frame, the compatibility of the connection with the Hermitian metric gives
\begin{eqnarray*}
\langle D\,e_i,e_j\rangle+\langle e_i,D\,e_j\rangle &=& \theta_{ji}+\bar{\theta}_{ij}\\
&=& 0\;\;\;\;\;\mbox{for}\; \;\;\;1\leq i,j\leq 2.
\end{eqnarray*}
For $1\leq i,j\leq 2$, equating $(1,0)$ and $(0,1)$ forms separately to zero in the equation 
$\theta_{ij}+\bar{\theta}_{ji}=0$, we obtain $\alpha_{11}=\partial(\log h)$, $\alpha_{12}=0$,
$\alpha_{21}= h^{1/2}\frac{\partial(h^{-1}\langle\gamma_1,\gamma_0\rangle)}
{(\|\gamma_1\|^2-\frac{|\langle\gamma_1,\gamma_0\rangle|^2}
{\|\gamma_0\|^2})^{1/2}}$ and $\alpha_{22}=\frac{1}{2}\frac{\partial(\|\gamma_1\|^2-\frac{\langle \gamma_1,\gamma_0\rangle}{\|\gamma_0\|^2})}
{(\|\gamma_1\|^2-\frac{\langle \gamma_1,\gamma_0\rangle}{\|\gamma_0\|^2})}$.
Hence the second fundamental form for the inclusion $E_0\subset E$ is given by the formula: 
\begin{eqnarray*}
\theta_{12}&=&  - h^{1/2}\frac{\bar{\partial}(h^{-1}\langle\gamma_1,\gamma_0\rangle)}
{\big (\|\gamma_1\|^2-\frac{|\langle\gamma_1,\gamma_0\rangle|^2}
{\|\gamma_0\|^2}\big )^{1/2}}
\,=\, - \frac{\frac{\partial^2}{\partial z\partial \bar z} \log h \,d \bar z}
{\big (\frac{\|t_1\|^2}{\|\gamma_0\|^2}+\frac{\partial^2}{\partial z \partial \bar z} \log h\big )^{1/2}}\\
\end{eqnarray*}
If $T=\begin{pmatrix}T_0 & S \\
0 & T_1 \\
\end{pmatrix}$
is an operator in $\mathcal{F}B_2(\Omega)$ and  $t_1$ is a non-vanishing holomorphic section of the vector bundle
$E_1$ corresponding to the operator $T_1,$ then we may assume, without loss of generality, that
$S(t_1)$ is a holomorphic frame of $E_0.$ The second fundamental form $\theta_{12}$ of the inclusion 
$E_0 \subseteq E$, in this case, is therefore equal to $$-\tfrac{\tfrac{\partial^2}{\partial
z  \partial \bar z}\log\|S(t_1)\|^2d\bar
z}{\big(\tfrac{\|t_1\|^2}{\|S(t_1)\|^2} + \tfrac{\partial^2}
{\partial z\partial \bar z}\log\|S(t_1)\|^2 \big)^{1/2}}.$$
It follows from Theorem \ref{maint} that the second fundamental form of the inclusion $E_0 \subseteq E$ and the curvature
of $E_1$ form a complete set of invariants for the operator $T.$ We restate Theorem \ref{maint} using the second fundamental form $\theta_{12}.$

\begin{thm}
Suppose that $T=\Big (\begin{smallmatrix}T_0 & S \\
0 & T_1 \\
\end{smallmatrix}\Big )$
and $\tilde{T}=\Big (\begin{smallmatrix}\tilde{T}_0 & \tilde S \\
0 & \tilde{T}_1 \\
\end{smallmatrix}\Big )$ are any two  operators in $\mathcal FB_2(\Omega).$ 
Then the operators $T$ and $\tilde{T}$ are unitarily equivalent if and only if $\mathcal{K}_{T_1}=\mathcal{K}_{\tilde{T}_1}$ (or
$\mathcal{K}_{T_0}=\mathcal{K}_{\tilde{T}_0}$) and $\theta_{12} = \tilde{\theta}_{12}.$
\end{thm}
\subsection{\large\sf  Application to homogeneous operators} We 
use the machinery developed here to list the unitary equivalence classes of  homogeneous operators in $B_n(\mathbb D),$ $n=2.$ For $n=1$ this was done in \cite{gm1} and in \cite{Wil} for $n=2.$ The classification of homogeneous operators in $B_n(\mathbb D)$ was given in \cite{AK} for an arbitrary $n.$  The proofs of \cite{Wil} and \cite{AK} use tools from Differential geometry and the representation theory of Lie groups respectively. While the description below is very close to the spirit of \cite{gm1}. 
\begin{defn}  An operator $T$ is said to be homogeneous if $\varphi(T)$ is unitarily equivalent to
$T$ for all $\varphi$ in M\"{o}b which are analytic on the spectrum
of $T$.
\end{defn}
\begin{prop}[\cite{gm1}]\label{propgm}
An operator $T$  in $\mathcal{B}_1(\mathbb D)$   is  homogeneous
if and only if $$\mathcal{K}_T(w)=-\lambda
(1-|w|^2)^{-2}$$ for some  positive real number $\lambda$.
\end{prop}
\begin{rem} From the Proposition \ref{propgm}, it follows that $T$ is unitarily equivalent to the adjoint of the multiplication operator ${M^{(\lambda)}}$ acting on the  reproducing
kernel Hilbert space $(\mathcal{H}^{(\lambda)},K^{(\lambda)}),$ 
where the reproducing kernel $K^{(\lambda)}$ is of the form $\frac{1}{(1-z\bar w)^{\lambda}},$ $z,w \in \mathbb D.$
\end{rem}

\begin{prop}\label{prop1}
Let $T$ be an  operator in 
$\mathcal{F}{B}_{2}(\mathbb{D})$ and let $t_1$ be a non-vanishing holomorphic  section of the bundle
$E_1$ corresponding to the operator $T_1.$ For any $\varphi$ in M\"{o}b,  set $t_{1, \varphi}=t_1o\varphi^{-1}.$   The operator $T$ is  homogeneous 
if and only if $T_0$, $T_1$ are homogeneous and
$\tfrac{\|S(t_{1,\varphi})\|^2}{\|t_{1,\varphi}\|^2}=
|(\varphi^{-1})^{\prime}|^2 \tfrac{\|S(t_1)\|^2}{\|t_1\|^2}$  for all
$\varphi$ in  M\"{o}b.
\end{prop}
\begin{proof}
Using the intertwining property in the class $\mathcal FB_2(\mathbb D),$ we see that  $$\varphi(T)=\begin{pmatrix}
\varphi(T_0) & S\varphi^{\prime}(T_1) \\
0 & \varphi(T_1) \\
\end{pmatrix}.$$
Suppose that $T$ is homogeneous, that is, $T$ is
unitarily equivalent to $\varphi(T)$ for $\varphi$ in M\"{o}b, at any rate, for some open subset of M\"{o}b.
From Theorem \ref{maint}, it follows that $T_0$ is
unitarily equivalent to $\varphi(T_0)$, $T_1$ is unitarily equivalent
to $\varphi(T_1)$ and
\begin{eqnarray}\label{e12}\tfrac{\|S\varphi^{\prime}(T_1)(t_{1,\varphi}(w))\|^2}
{\|t_{1,\varphi}(w)\|^2}=\tfrac{\|S(t_1(w))\|^2}{\|t_1(w)\|^2}.
\end{eqnarray}
Now, we have 
\begin{eqnarray}
\tfrac{\|S\,\varphi^{\prime}(T_1)(t_{1,\varphi}(w))\|^2}
{\|t_{1,\varphi}(w)\|^2}&=&\tfrac{\|S\varphi^{\prime}(\varphi^{-1}(w))(t_{1,\varphi}(w))\|^2}{\|t_{1,\varphi}(w)\|^2}\nonumber\\
&=& \tfrac{|\varphi^{\prime}(\varphi^{-1}(w))|^2\|S(t_{1,\varphi}(w))\|^2}{\|t_{1,\varphi}(w)\|^2}\nonumber\\
&=&
\tfrac{|(\varphi^{-1})^{\prime}(w)|^{-2}\|S(t_{1,\varphi}(w))\|^2}{\|t_{1,\varphi}(w)\|^2}\label{e13}.
\end{eqnarray}
From equations (\ref{e12}) and (\ref{e13}), it follows that
 \begin{eqnarray}\label{e14}\tfrac{\|S(t_{1,\varphi}(w))\|^2}{\|t_{1,\varphi}(w)\|^2}= |(\varphi^{-1})^{\prime}(w)|^2 \tfrac{\|S(t_1(w))\|^2}{\|t_1(w)\|^2}.\end{eqnarray}

Conversely suppose that $T_0$, $T_1$ are homogeneous operators and
$$\tfrac{\|S(t_{1,\varphi}(w))\|^2}{\|t_{1,\varphi}(w)\|^2}=
|(\varphi^{-1})^{\prime}(w)|^2
\tfrac{\|S(t_1(w))\|^2}{\|t_1(w)\|^2}$$ for all $\varphi$ in M\"{o}b. From
equations (\ref{e13}), (\ref{e14}) and Theorem \ref{maint}, it
follows that $T$ is a homogeneous operator.
\end{proof}

\begin{cor}An  operator $T$ in  $\mathcal{F}{B}_{2}(\mathbb{D})$ is a homogeneous if and only if
\begin{enumerate}
\item[(i)] $T_0$ and $T_1$ are homogeneous operators; \item[(ii)]
$\mathcal{K}_{T_1}(w)=\mathcal{K}_{T_0}(w)+\mathcal{K}_{B^*}(w),
\;w\in\mathbb{D},$ where $B$ is the forward Bergman shift;

\item[(iii)] $S(t_1(w))=\alpha \gamma_0(w)$ for some positive real number $\alpha$ and $\|t_1(w)\|^2=\tfrac{1}{(1-|w|^2)^{\lambda+2}}$, $\|\gamma_0(w)\|^2=\tfrac{1}{(1-|w|^2)^{\lambda}}$.
\end{enumerate}
\end{cor}

\begin{proof}
Suppose $T$ is a homogeneous operator.  Proposition
\ref{prop1} shows that $T_0$ and $T_1$ are homogeneous operators. We may therefore find non-vanishing holomorphic sections $\gamma_0$ and $t_1$
of $E_0$ and $E_1,$ respectively, such that
$\|\gamma_0(w)\|^2=(1-|w|^2)^{-\lambda}$ and
$\|t_1(w)\|^2=(1-|w|^2)^{-\mu}$ for some positive real $\lambda$ and $\mu.$ For $\varphi$ in M\"{o}b,  set $\gamma_{0,{\varphi}}=\gamma_0\circ\varphi^{-1}$ and
$t_{1,\varphi}=t_1\circ\varphi^{-1}.$ Clearly
$\|\gamma_{0,{\varphi}}(w)\|^2=|(\varphi^{-1})^{\prime}(w)|^{-\lambda}\|\gamma_0(w)\|^2$
and
$\|t_{1,\varphi}(w)\|^2=|(\varphi^{-1})^{\prime}(w)|^{-\mu}\|t_1(w)\|^2$.
Let $S(t_1(w))=\psi(w)\gamma_0(w)$ for some holomorphic function $\psi$  on $\mathbb{D}.$ We have 
$S(t_{1,\varphi}(w))=S(t_1(\varphi^{-1}(w)))=\psi(\varphi^{-1}(w))\gamma_0(\varphi^{-1}(w))=
\psi(\varphi^{-1}(w))\gamma_{0,{\varphi}}(w)$ and 
\begin{eqnarray}\label{e15}\tfrac{\|S(t_{1,\varphi}(w))\|^2}{\|t_{1,\varphi}(w)\|^2}= |(\varphi^{-1})^{\prime}(w)|^2 \tfrac{\|S(t_1(w))\|^2}{\|t_1(w)\|^2}.
\end{eqnarray}
Combining these we see that 
\begin{eqnarray}
\tfrac{\|S(t_{1,\varphi}(w))\|^2}{\|t_{1,\varphi}(w)\|^2}&=& |\psi(\varphi^{-1}(w))|^2 \tfrac{\|(\gamma_{0,{\varphi}}(w))\|^2}{\|t_{1,\varphi}(w)\|^2}\nonumber\\
&=&|\psi(\varphi^{-1}(w))|^2
|(\varphi^{-1})^{\prime}(w)|^{\mu-\lambda}\tfrac{\|(\gamma_0(w))\|^2}{\|t_1(w)\|^2}\label{e16}.
\end{eqnarray}
From the equations (\ref{e15}) and (\ref{e16}), we get
\begin{eqnarray}\label{e17}
|{\psi(w)}|^2|(\varphi^{-1})^{\prime}(w)|^{\lambda+2-\mu}=
|{\psi(\varphi^{-1}w)}|^2
\end{eqnarray}
Pick $\varphi=\varphi_u,$ where $\varphi_u(w)=\tfrac{w-u}{1-\bar u w}$ and put $w=0$ in the equation (\ref{e17}). Then 
\begin{eqnarray}
|{\psi(0)}|^2(1-|u|^2)^{\lambda+2-\mu}=|{\psi(u)}|^2\label{e18}.
\end{eqnarray}
If $\psi(0)=0$ then  equation (\ref{e18}) implies that $\psi(u)=0$
for all $u\in\mathbb{D},$ which makes $S=0$ leading to a contradiction. Thus $\psi(0)\neq 0$. Differentiating of both sides the equation \eqref{e18},  we see that 
$$(\lambda+2-\mu)\tfrac{\partial^2}{\partial u\partial\bar u}\log(1-|u|^2)=0.$$
Hence we conclude that $\mu=\lambda+2.$ Putting $\mu=\lambda+2$
in the equation \eqref{e18} we find that $\psi$ must be a constant
function. Hence there is a constant $\alpha$ such that $S(t_1(w))=\alpha \gamma_0(w)$ for all $w\in
\Omega.$  Finally,
\begin{eqnarray*}
\mathcal{K}_{T_1}(w)&=& \bar {\partial}\partial \log\|t_1(w)\|^2\\
&=& \bar {\partial}\partial \log(1-|w|^2)^{-\mu}\\
&=& \bar {\partial}\partial \log(1-|w|^2)^{-\lambda-2}\\
&=& \bar {\partial}\partial \log(1-|w|^2)^{-\lambda}+\bar {\partial}\partial \log(1-|w|^2)^{-2}\\
&=&\bar {\partial}\partial \log\|\gamma_0(w)\|^2+\bar {\partial}\partial \log(1-|w|^2)^{-2}\\
&=&\mathcal{K}_{T_0}(w)+\mathcal{K}_{B^*}(w).
\end{eqnarray*}
Conversely, suppose that conditions $(i),(ii)$ and $(iii)$ are
met. We need to show that $T$ is a homogeneous operator.
Condition $(ii)$ is equivalent to $\mu=\lambda+2$. By Proposition
\ref{prop1}, it is sufficient to show that
$$\tfrac{\|S(t_{1,\varphi}(w))\|^2}{\|t_{1,\varphi}(w)\|^2}= |(\varphi^{-1})^{\prime}(w)|^2 \tfrac{\|S(t_1(w))\|^2}{\|t_1(w)\|^2}.$$
However, we have 
\begin{eqnarray*}
\tfrac{\|S(t_{1,\varphi}(w))\|^2}{\|t_{1,\varphi}(w)\|^2}&=&|\alpha|^2 \tfrac{\|(\gamma_{0,{\varphi}}(w))\|^2}{\|t_{1,\varphi}(w)\|^2}\\
&=& |\alpha|^2 |(\varphi^{-1})^{\prime}(w)|^{\mu-\lambda}\tfrac{\|(\gamma_0(w))\|^2}{\|t_1(w)\|^2}\\
&=&|\alpha|^2 |(\varphi^{-1})^{\prime}(w)|^{2}\tfrac{\|(\gamma_0(w))\|^2}{\|t_1(w)\|^2}\\
&=& |(\varphi^{-1})^{\prime}(w)|^{2}
\tfrac{\|S(t_1(w))\|^2}{\|t_1(w)\|^2}.
\end{eqnarray*}
\end{proof}
\subsection{\large\sf Irreducibility and strong irreducibility in $\mathcal FB_2(\Omega)$} In this subsection, we show that an operator $T$ in $\mathcal FB_2(\Omega)$ is irreducible.
Furthermore, if the intertwining operator $S$ is invertible, then $T$ is strongly irreducible. (Recall that an operator $T$ is said to be strongly irreducible if the commutant $\{T\}^\prime$ of the operator $T$ contains no idempotent operator.) We also provide a more direct proof of proposition \ref{mainp}, which easily generalizes to the case of an arbitrary $n.$

\begin{defn} Let  $T_1$ and $T_2$ be any two bounded linear operators on the Hilbert space  ${\mathcal H}.$  Define  $\sigma_{T_1,T_2}:{\mathcal L}({\mathcal H})\rightarrow
{\mathcal L}({\mathcal H})$ to be the operator 
$$\sigma_{T_1,T_2}(X)=T_{1}X-XT_{2},\,\, X\in {\mathcal L}({\mathcal
H}).$$ Let $\sigma_{T}:{\mathcal L}({\mathcal
H})\rightarrow {\mathcal L}({\mathcal H})$  be the operator $\sigma_{T, T}.$
\end{defn}
Recall that an operator $T$ defined on a Hilbert space $\mathcal{H}$ is said to be 
quasi-nilpotent if $\lim_{n\to \infty}\|T^n\|^{1/n}=0.$
\begin{lem}\label{qnl}
Suppose $T$ is in $B_1(\Omega)$ and $X$ is a quasi-nilpotent operator such
that $TX=XT.$ Then $X=0.$
\end{lem}
\begin{proof}
Let $\gamma$ be a non-vanishing holomorphic section for $E_T.$ Since $TX=XT,$ we see that 
$X(\gamma)$ is also a  holomorphic  section of $E_T.$ Hence
$X(\gamma(w))=\phi(w)\gamma(w)$ for some holomorphic function $\phi$ defined on $\Omega$. Clearly, $X^n(\gamma(w))=\phi(w)^n\gamma(w).$
Now, we have 
\begin{eqnarray*}
|\phi(w)|^n\|\gamma(w)\|&=& \|\phi(w)^n\gamma(w)\|\\
&=&\|X^n(\gamma(w))\|\\
&\leq&\|X^n\|\|\gamma(w)\|
\end{eqnarray*}
Thus, for $n\in\mathbb{N}$ and
$w\in\Omega,$ we have $|\phi(w)|\leq \|X^n\|^{1/n}$ implying $\phi(w)=0,$ $w\in\Omega.
$ Hence $X=0.$
\end{proof}
The following theorem from \cite{hel} is the key to an alternative proof of the proposition \ref{mainp} and its generalization in the following section.
\begin{thm}  \label{hlem} Let $P, T$ be two bounded linear operators. 
If $P\in \mbox{ran}\,\sigma_{T}\cap\mbox{ker}\,\sigma_{T},$ then
$P$ is a quasi-nilpotent.
\end{thm}
\noindent {\sf A second Proof of Proposition \ref{mainp}}
\begin{proof}
Suppose $T$ is unitarily equivalent to $\tilde{T}$ via the unitary $U,$ namely, $UT=TU.$ Then  
\begin{eqnarray}
U_{21}S+U_{22}T_1 &=&\tilde{T}_1U_{22} \label{e1}\\
U_{21}T_0&=&\tilde{T}_1 U_{21}.\label{e2}
\end{eqnarray}
Equivalently, we also have $TU^*=U^*\tilde T,$ which gives an additional relationship:
\begin{eqnarray}
T_1U_{12}^*&=&U_{12}^*\tilde{T}_0. \label{e3}
\end{eqnarray}
Using these equations, we compute 
\begin{eqnarray*}
U_{21}S U_{12}^*\tilde{S}&=&(\tilde T_1 U_{22}-U_{22}{T}_1)U_{12}^*\tilde S \\
&=&\tilde T_1 U_{22}U_{12}^*\tilde S- U_{22}{T}_1 U_{12}^*\tilde S\\
&=&\tilde T_1 U_{22}U_{12}^*\tilde S-U_{22}U_{12}^*\tilde T_0 \tilde S\\  \:\;\;\;\;\;\;\;\;\;\;
&=&\tilde T_1 U_{22}U_{12}^*\tilde S-U_{22}U_{12}^* \tilde S\tilde T_1\\
&=& \sigma_{\tilde T_1}(U_{22}U_{12}^*\tilde S),
\end{eqnarray*}
and
\begin{eqnarray*}
(U_{21}S U_{12}^*\tilde{S})\tilde T_1&=& U_{12}S U_{12}^*\tilde T_0\tilde S\\
&=& U_{21}{S}{T}_1U_{12}^*\tilde S\\
&=& U_{21}{T}_0{S}U_{12}^*\tilde S\\
&=& \tilde T_1(U_{12}\tilde{S}U_{12}^*\tilde S).\\
\end{eqnarray*}
Thus $U_{21}S U_{12}^*\tilde{S}\in \mbox{ran}\,\sigma_{\tilde
T_1}\cap\mbox{ker}\,\sigma_{\tilde T_1}$. From Lemma \ref{qnl}
and Theorem \ref{hlem}, it follows that  $$U_{21}S
U_{12}^*\tilde{S}=0.$$ Since $\tilde{S}$ has dense range, we 
have $U_{21}S U_{12}^*=0$. Let us consider the two possibilities for $U_{12}^*,$ namely, either $U_{12}^*=0$ or $U_{12}^*\neq0.$
If $U_{12}^*\neq0,$ then from equation (\ref{e3}), $U_{12}^*$ must 
have dense range. Since $S$ also has dense range, we have
$U_{21}=0.$ To complete the proof, we consider two cases.

{\sf Case 1:} Suppose $U_{21}=0$. In this case, we have to prove that
$U_{12}=0$. From $U^*U=I$, we get $U_{11}^*U_{11}=I$ and
$U_{12}^*U_{11}=0$. From $UT=\tilde T U$, we get $U_{11}T_0=\tilde
T_0 U_{11}$, so $U_{11}$ has dense rang. Since $U_{11}$ is an
isometry and has dense range, it follows that $U_{11}$ is onto. Hence $U_{11}$
is unitary. Since $U_{11}$ is unitary and $U_{12}^*U_{11}=0$, it follows that $U_{12}=0$.

{\sf Case 2:} Suppose $U_{12}=0$. In this case, we have to prove that $U_{21}=0.$ We have $U_{11}U_{11}^*=I$ and
$U_{21}U_{11}^*=0.$ The intertwining relation  $TU^*=U^* \tilde T $ gives 
$T_0U_{11}^*=U_{11}^*\tilde{T}_0.$  So $U_{11}^*$ has dense range.
Since $U_{11}^*$ is an isometry and it has dense range, we must conclude that  $U_{11}^*$
is onto. Hence $U_{11}$ is unitary and we have 
$U_{21}U_{11}^*=0$ forcing $U_{21}$ to be the $0$ operator.
\end{proof}

\begin{prop}\label{pirre}
Any operator $T$ in $\mathcal{F}B_2(\Omega)$ is irreducible. Also, if $T=\Big (\begin{smallmatrix}T_0 & I \\ 0&T_0\end{smallmatrix}\Big ),$ then it is strongly irreducible.  
\end{prop}
\begin{proof}
Let $P=(P_{ij})_{2\times 2}$ be a projection in the commutant $\{T\}^{\prime}$ of the operator $T,$ that is,
$$\begin{pmatrix}
P_{11} & P_{12} \\
P_{21} & P_{22} \\
\end{pmatrix}
\begin{pmatrix}
T_0 & S \\
0 & T_1 \\
\end{pmatrix}=\begin{pmatrix}
T_0 & S \\
0 & T_1 \\
\end{pmatrix}\begin{pmatrix}
P_{11} & P_{12} \\
P_{21} & P_{22} \\
\end{pmatrix}.$$
This equality implies that $P_{11}T_0= T_0P_{11}+SP_{21}$,
$P_{11}S+P_{12}T_1= T_0P_{12}+SP_{22}$, $P_{21}T_0=T_1P_{21}$ and
$P_{21}S+P_{22}T_1= T_1P_{22}$. Now
\begin{eqnarray*}
(P_{21}S)T_1&=&P_{21}(ST_1)
\,=\,P_{21}(T_0S)
\,=\,(P_{21}T_0)S
\,=\,T_1(P_{21}S).
\end{eqnarray*}
Thus $P_{21}S\in \mbox{ker}\;\sigma_{T_1}$. Also note that
\begin{eqnarray*}
P_{21}S&=&T_1P_{22}-P_{22}T_1
\,=\,\sigma_{T_1}(P_{22}).
\end{eqnarray*}
Hence $P_{21}S\in
\mbox{ran}\,\sigma_{T_1}\cap\mbox{ker}\,\sigma_{T_1}$. Thus from 
Lemma \ref{qnl} and Theorem \ref{hlem}, it follows that $P_{21}S=0$.
The operator $P_{21}$ must be $0$ since $S$ has dense range. 

To prove the first statement, we may assume that the operator $P$ is self adjoint and conclude  $P_{12}$ is $0$ as well.
Since both the operators $T_0$ and $T_1$ are irreducible and the projection $P$ is diagonal, it follows that $T$ must be irreducible. 

For the proof of the second statement, note that if $P$ is an idempotent of the form  $\Big ( \begin{smallmatrix}P_{11} & P_{12}\\
0 & P_{22} \end{smallmatrix}\Big )$, both $P_{11}$ and $P_{22}$ must be
idempotents. By our hypothesis, $P_{11}$ and $P_{22}$ must also
commute with $T_0$, which is strongly irreducible, hence
$P_{11}=0\;\mbox{or}\,I$ and $P_{22}=0\;\mbox{or}\,I$. By using
Theorem 2.19, we see that if $P= \Big ( \begin{smallmatrix}I & P_{12}\\
0 & 0 \end{smallmatrix}\Big )$ or $P= \Big ( \begin{smallmatrix}0 & P_{12}\\
0 & I \end{smallmatrix}\Big ),$ then $P$ does not commute with $ \Big ( \begin{smallmatrix}T_0 & I\\
0 & T_0 \end{smallmatrix}\Big )$. Thus $P= \Big ( \begin{smallmatrix}I & P_{12}\\
0 & I \end{smallmatrix}\Big )$ or $P= \Big ( \begin{smallmatrix}0 & P_{12}\\
0 & 0 \end{smallmatrix}\Big )$. Now, using the equation $P^2=P$, we
conclude that $P_{12}$ must be zero. Thus $P=I$ or 
$P=0$.

\end{proof}
We now give a sufficient condition for an operator $T$ in $\mathcal FB_2(\Omega)$ to be strongly irreducible.  
\begin{prop}\label{psirre}
  Let  $T=\begin{pmatrix}
                                                 T_0 & S \\
                                                  0 & T_1 \\
                                                \end{pmatrix}$ be
                                                an operator in $\mathcal{F}B_2(\Omega).$
If the operator $S$ is invertible, then the operator $T$ is
strongly irreducible.
\end{prop}
\begin{proof}
By our hypothesis, the operator 
$X=\begin{pmatrix}
I & 0 \\
0 & S \\
\end{pmatrix}
$ is invertible.  Now
\begin{eqnarray*}
XTX^{-1}&=&\begin{pmatrix}
I & 0 \\
0 & S \\
\end{pmatrix}\begin{pmatrix}
T_0 & S \\
0 & T_1 \\
\end{pmatrix}\begin{pmatrix}
I & 0 \\
0 & S \\
\end{pmatrix}^{-1}\\
&=&\begin{pmatrix}
T_0 & I \\
0 & ST_1S^{-1} \\
\end{pmatrix}\\
&=&\begin{pmatrix}
T_0 & I \\
0 & T_0 \\
\end{pmatrix}.
\end{eqnarray*}
Thus $T$ is similar to  a strongly irreducible operator and consequently it is strongly irreducible.
\end{proof}
We conclude this section with a characterization of strong irreducibility in $\mathcal FB_2(\Omega).$

\begin{prop}\label{psirre1}
An operator  
$T=\begin{pmatrix}
T_0 & S \\
0 & T_1 
\end{pmatrix}$ 
in $\mathcal F B_2(\Omega)$ is strongly irreducible if and only if $S\notin
\mbox{ran}\;\sigma_{T_0,T_1}$.
\end{prop}
\begin{proof}
Let $P$ be an idempotent in the commutant $\{T\}^{\prime}$ of the operator $T$.  The proof of the 
Proposition \ref{pirre} shows that $P$ must be upper triangular: $\begin{pmatrix}
P_{11} & P_{12} \\
0 & P_{22} \\
\end{pmatrix}$. The commutation relation $PT=TP$ gives us $P_{11}T_0= T_0P_{11}$, $P_{22}T_1= T_1P_{22}$
and \begin{eqnarray}\label{esirre1}P_{11}S-S P_{22}= T_0P_{12}-P_{12}T_1.\end{eqnarray}
Since $P_{i+1i+1}\in\{T_{i}\}^{\prime}$ for $0\leq i\leq 1$, it follows that 
$P_{ii}$ can be either $I$ or $0$. If either $P_{11}=I$ and 
$P_{22}=0$ or $P_{11}=0$ and $P_{22}=I$, then $S$ is in
$\mbox{ran}\;\sigma_{T_0,T_1}$ contradicting  our
assumption. Thus $P$ is of the form $\begin{pmatrix}
I & P_{12} \\
0 & I \\ \end{pmatrix}$ or $\begin{pmatrix}
0 & P_{12} \\
0 & 0 \\
\end{pmatrix}.$ Since $P$ is an idempotent operator, we must have $P_{12}=0$. Hence $T$ is strongly irreducible. 

Assume that the operator  $S$ is in $\mbox{ran}\;\sigma_{T_0,T_1}.$ In this case, we show that $T$ cannot be strongly irreducible completing the proof. Since \mbox{$S\in
\mbox{ran}\;\sigma_{T_0,T_1}$}, we can find an operator
$P_{12}$ such that
\begin{eqnarray}\label{esirre2}S&=&\sigma_{T_0,T_1}(P_{12})\nonumber\\
&=&T_0P_{12}-P_{12}T_1.\end{eqnarray} The operator
$P=\begin{pmatrix}
I & P_{12} \\
0 & 0 \\ \end{pmatrix}$  is an idempotent operator. We have 
\begin{eqnarray}\label{esirre3}
\begin{pmatrix}
I & P_{12} \\
0 & 0 \\ \end{pmatrix}\begin{pmatrix}
T_0 & S \\
0 & T_1 \\ \end{pmatrix}=\begin{pmatrix}
T_0 & S+P_{12}T_1 \\
0 & 0 \\ \end{pmatrix}
\end{eqnarray}
and \begin{eqnarray}\label{esirre4}
\begin{pmatrix}
T_0 & S \\
0 & T_1 \\ \end{pmatrix}\begin{pmatrix}
I & P_{12} \\
0 & 0 \\ \end{pmatrix}=\begin{pmatrix}
T_0 & T_0 P_{12} \\
0 & 0 \\ \end{pmatrix}.
\end{eqnarray}
From these equations, 
we have $PT=TP$ proving that the operator $T$ is not strongly irreducible.
\end{proof}

\section{Rigidity of the flag structure}
There are two possible generalizations of the class $\mathcal FB_2(\Omega)$ to operators in $B_n(\Omega)$ for an arbitrary $n\in \mathbb N.$ One of these is more restrictive but has the advantage of possessing a tractable set of complete unitary invariants. In both cases, the inherent flag structure is rigid as will be seen below. 
\begin{defn} We let $\mathcal{F}B_n(\Omega)$ be the set of all bounded linear operators $T$ defined on some complex separable Hilbert space $\mathcal H = \mathcal H_0 \oplus \cdots \oplus \mathcal H_{n-1},$ which are of the form 
$$ T=\begin{pmatrix}
T_{0} & S_{0,1} & S_{0,2}&\cdots&S_{0,n-1}\\
0 &T_{1}&S_{1,2}&\cdots&S_{1,n-1} \\
\vdots&\ddots &\ddots&\ddots&\vdots\\
0&\cdots&0&T_{n-2}&S_{n-2,n-1}\\
0&\cdots&\cdots&0&T_{n-1}\\
\end{pmatrix},$$
where the operator $T_i:\mathcal H_i \to \mathcal H_i,$ defined on the complex separable Hilbert space $\mathcal H_i,$ $0\leq i \leq n-1,$ is assumed to be in $B_1(\Omega)$ and  $S_{i,i+1}:\mathcal H_{i+1} \to \mathcal H_i,$ is assumed to be a non-zero intertwining operator, namely, $T_iS_{i,i+1}=S_{i,i+1}T_{i+1},$  $0\leq i \leq n-2.$
\end{defn}
Even without mandating the intertwining condition, the set of operators described above belong to the Cowen-Douglas class $B_n(\Omega).$ An inductive proof presents no difficulty starting with base case of $n=2,$ which was proved in the previous section. 
Therefore, in particular, $\mathcal FB_n(\Omega) \subseteq B_n(\Omega).$ The proof is a straightforward induction on $n.$ 
The following proposition is the first step in the proof of the rigidity theorem.
\begin{prop}\label{dp}
If $X$ is an invertible operator intertwining two operators  $T$ and $\tilde{T}$ from $\mathcal{F}B_n(\Omega)$, then $X$ is upper  triangular.
\end{prop}
\begin{proof}
The proof is by induction on $n.$ To begin the induction, for $n=2,$ following the method of the proof in Proposition \ref{pirre}, we see that an invertible intertwining operator between two operators in $\mathcal FB_2(\Omega)$ must be upper triangular. Now, assume that any invertible intertwiner $X$ between two operators in $\mathcal F B_k(\Omega)$ is upper triangular for all $k< n.$  Let $Y=X^{-1}$ and $X=\big (\!\!\big (X_{i,j}\big )\!\!\big )_{n\times n},$  $Y=\big (\!\!(Y_{i,j}\big )\!\!\big )_{n\times n}$ be the block decompositions of the two operators $X$ and $Y,$ respectively.

{\sf Step 1: To show that $X_{n,1}=0$ or $Y_{n,1}=0$.}
We have that $XT = \tilde{T}X$ and consequently 
\begin{eqnarray}\label{he1} X_{n,1}T_0=\tilde{T}_{n-1}X_{n,1},\,\,X_{n,1}S_{0,1}+X_{n,2}T_{1}=\tilde{T}_{n-1}X_{n,2}.\end{eqnarray}
Since $T_{k}S_{k,k+1}=S_{k,k+1}T_{k+1}$ for $k=0,1,2,\cdots n-1,$ multiplying the second equation in \eqref{he1} by $S_{1,2}\cdots S_{n-2,n-1},$ and replacing $T_{1}S_{1,2}\cdots
S_{n-2,n-1}$ with $S_{1,2}\cdots
S_{n-2,n-1}T_{n-1},$ we have 
\begin{equation}\label{3.11}
\begin{array}{lllll}
X_{n,1}S_{0,1}S_{1,2}\cdots S_{n-2,n-1}+X_{n,2}S_{1,2}\cdots
S_{n-2,n-1}T_{n-1}&=&\tilde{T}_{n-1}X_{n,2}S_{1,2}\cdots
S_{n-2,n-1}. 
\end{array}
\end{equation}
We also have $TY=Y\tilde{T}$  giving us the set of equations  
\begin{eqnarray}\label{he2}T_{n-1}Y_{n,1}=Y_{n,1}\tilde{T}_0,\,\,  \tilde{T}_{k}\tilde{S}_{k,k+1}=\widetilde{S}_{k,k+1}\tilde{T}_{k+1},
\,\, k=0,1,2,\cdots n.\end{eqnarray}  
Now, multiply both sides of the equation \eqref{3.11} by $Y_{n,1},$ using the commutation $T_{n-1}Y_{n,1} = Y_{n,1}\tilde{T}_0,$ then again multiplying both sides of the resulting equation by $\widetilde{S}_{0,1}\cdots \widetilde{S}_{n-2,n-1}$ and finally using the commutation relations \eqref{he2}, we have 
\begin{align}\label{3.13}
X_{n,1}S_{0,1}S_{1,2}\cdots
S_{n-2,n-1}Y_{n,1}\tilde{S}_{0,1}\cdots
\tilde{S}_{n-2,n-1}+X_{n,2}S_{1,2}\cdots
S_{n-2,n-1}Y_{n,1}\tilde{S}_{0,1}\cdots
\tilde{S}_{n-2,n-1}\tilde{T}_{n-1}\nonumber\\
=\tilde{T}_{n-1}X_{n,2}S_{1,2}\cdots
S_{n-2,n-1}Y_{n,1}\tilde{S}_{0,1}\cdots \tilde{S}_{n-2,n-1}.&
\end{align}
Therefore, we see that $$X_{n,1}S_{0,1}S_{1,2}\cdots
S_{n-2,n-1}Y_{n,1}\tilde{S}_{0,1}\cdots
\tilde{S}_{n-2,n-1}$$ is in the range of the operator 
$\sigma_{\tilde{T}_{n-1}}.$
Indeed it is also in the kernel of $\sigma_{\tilde{T}_{n-1}},$ as is evident from the following string of equalities: 
\begin{eqnarray*}
X_{n,1}S_{0,1}\cdots
S_{n-2,n-1}Y_{n,1}\tilde{S}_{0,1}\cdots
\tilde{S}_{n-2,n-1}\tilde{T}_{n-1}
&=&X_{n,1}S_{0,1}\cdots
S_{n-2,n-1}Y_{n,1}\tilde{T}_0\tilde{S}_{0,1}\cdots \tilde{S}_{n-2,n-1}\\
&=&X_{n,1}S_{0,1}\cdots
S_{n-2,n-1}T_{n-1}Y_{n,1}\tilde{S}_{0,1}\cdots \tilde{S}_{n-2,n-1}\\
&=&X_{n,1}T_0S_{0,1}\cdots
S_{n-2,n-1}Y_{n,1}\tilde{S}_{0,1}\cdots \tilde{S}_{n-2,n-1}\\
&=&\tilde{T}_{n-1}X_{n,1}S_{0,1}\cdots
S_{n-2,n-1}Y_{n,1}\tilde{S}_{0,1}\cdots
\tilde{S}_{n-2,n-1}.
\end{eqnarray*}
Thus
$$X_{n,1}S_{0,1}S_{1,2}\cdots S_{n-2,n-1}Y_{n,1}\tilde{S}_{0,1}\cdots
\tilde{S}_{n-2,n-1}\in \ker \sigma_{\tilde{T}_{n-1}}\cap
\mbox{ran}\,{\sigma_{\tilde{T}_{n-1}}}.$$
Consequently, using Lemma \ref{qnl} and Theorem \ref{hlem}, we conclude that 
$$X_{n,1}S_{0,1}S_{1,2}\cdots
S_{n-2,n-1}Y_{n,1}\tilde{S}_{0,1}\cdots
\tilde{S}_{n-2,n-1}=0.$$

By hypothesis, all the operators $S_{k,k+1}, \tilde{S}_{k,k+1}, k=0,2,\cdots n-2$ have dense range. If $Y_{n,1}\neq 0,$ then  equation (\ref{he2}) and
Proposition \ref{dr} ensure that  $Y_{n,1}$ has  dense range.  Hence $X_{n,1}=0.$  

{\sf Step 2: For $0<i<n,$ we have $X_{n,i}=0$ or $Y_{n,i}=0.$} 
Assume that $Y_{n,1}$ has dense range and $X_{n,1}=0.$ In this case, 
\begin{eqnarray}\label{he3}X_{n,2}T_1 =\tilde{T}_{n-1} X_{n,2},\,\, X_{n,2}S_{1,2}+X_{n,3}T_{2}=\tilde{T}_{n-1}X_{n,3}.
\end{eqnarray} 
As in the proof of Step 1, we have 
\begin{eqnarray}\label{he4}
X_{n,2}S_{1,2}S_{2,3}\cdots S_{n-2,n-1}+X_{n,3}S_{2,3}\cdots
S_{n-2,n-1}T_{n-1}&=&\tilde{T}_{n-1}X_{n,3}S_{2,3}\cdots
S_{n-2,n-1} .\end{eqnarray}
Computations as in the proof of Step 1, using equation (\ref{he4}), show that 
%
$$X_{n,2}S_{1,2}S_{2,3}\cdots S_{n-2,n-1}Y_{n,1}\tilde{S}_{0,1}\cdots
\tilde{S}_{n-2,n-1}\in
\mbox{ran}_{\sigma_{\tilde{T}_{n-1}}}\cap \ker \sigma_{\tilde{T}_{n-1}}.$$ 
Since $Y_{n,1}$ has dense range,
it follows that $X_{n,2}=0.$
For  $i<n-1,$ we also have 
$$X_{n,i}S_{i-1,i}S_{i,i+1}\cdots S_{n-2,n-1}Y_{n,1}\tilde{S}_{0,1}\cdots
\tilde{S}_{n-2,n-1}\in
\mbox{ran}_{\sigma_{\tilde{T}_{n-1}}}\cap \ker \sigma_{\tilde{T}_{n-1}}.$$ 
again, since $Y_{n,1}$ has  dense range, it follows that
$X_{n,i}=0,$ for all $i<n-1.$

Let us write the operator $X,$  in the form of a $2\times 2$ block matrix as $\begin{pmatrix} X_{n-1\times n-1} &  X_{n-1\times 1}\\ 0 & X_{n,n} \end{pmatrix},$ where if $X= \big (\!\big( X_{i,j}\big )\!\!\big)_{i,j=1}^n,$ then $X_{n-1\times n-1}$ is the operator $\big (\!\big( X_{i,j}\big )\!\!\big)_{i,j=1}^{n-1}$ and $X_{n-1\times 1}$ is the operator $\big (\!\big( X_{i\,n} \big )\!\!\big)_{i=1}^{n-1}.$  We assign a similar meaning to  the operators $T_{n-1\times n-1}$ and $\tilde{T}_{n-1\times n-1}$ after writing the operators $T$ and $\tilde{T}$ in the form of $2\times 2$ block matrices with respect to the same decomposition as of the operator $X$. The  $(2,1)$ entry in these block matrices is $0.$ By assumption, we have $XT =\tilde{T} X,$ which shows that $X_{n-1\times n-1}T_{n-1\times n-1} = \tilde{T}_{n-1\times n-1} X_{n-1\times n-1}.$ Now, the induction hypothesis guarantees that $X_{n-1\times n-1}$ must be upper triangular completing the proof.

If $X$ is an upper triangular block matrix, then 
$Y=X^{-1}$ must also be upper triangular. In fact, since
$XT=\tilde{T}X$ and $X$ is upper triangular, we have that
$X_{ii}T_{n-i}=\tilde{T}_{n-i}X_{ii}$ for all $i \leq n,$ consequently, $X_{ii}$ has dense range. Since $XY=YX=I,$ an
easy computation shows that $X_{nn}Y_{n,i}=0, X_{nn}Y_{nn}=I.$ It
follows that $Y_{n,i}=0.$ Then $Y$ is also seen to be upper triangular as in the proof of Step 2. 
\end{proof}
It is much easier to show that an operator in the commutant of $T\in \mathcal FB_n(\Omega)$ is upper triangular.   
\begin{prop}\label{utp}
Suppose $T$ is in $\mathcal{F}B_n(\Omega)$ and $X$ is a
bounded linear operator in the commutant of $T.$ Then $X$ is upper triangular.
\end{prop}
\begin{proof} First we prove that $X_{ni}=0$ for $1\leq i\leq n-1.$
Since $XT=TX,$ we see that
\begin{eqnarray}\label{ute}
X_{n1}T_0=T_{n-1}X_{n1}\;\mbox{and}\;\sum_{k=1}^{i}\big (X_{nk}S_{k-1,i}+X_{n\, i+1}T_i \big ) = T_{n-1}X_{n\, i+1}\,\mbox{for}\,1\leq i\leq n-1.
\end{eqnarray}
From equation (\ref{ute}), putting $i=1,$ we have
\begin{eqnarray}
X_{n1} S_{0,1}S_{1,2}\ldots S_{n-2,n-1}\in \ker\sigma_{T_{n-1}},\label{ute1}\\
X_{n1} S_{0,1}S_{1,2}\ldots S_{n-2,n-1}=
\sigma_{T_{n-1}}(X_{n2}S_{1,2}S_{2,3}\ldots
S_{n-2,n-1})\label{ute2}.
\end{eqnarray}
Therefore $X_{n1} S_{0,1}S_{1,2}\ldots S_{n-2,n-1}$ is in $\mbox{ran}\, \sigma_{T_{n-1}} \cap \ker\sigma_{T_{n-1}}.$   Combining Proposition \ref{dr} with 
Lemma \ref{qnl} and Theorem \ref{hlem}, we conclude that 
$X_{n1}=0.$
For $i=2$, making use of $X_{n1}=0$ in equations (\ref{ute}), we have
\begin{eqnarray}\label{ute4}X_{n2}S_{1,2}S_{2,3}\ldots
S_{n-2,n-1}\in\ker\sigma_{T_{n-1}}\cap
\mbox{ran}\;\sigma_{T_{n-1}}\end{eqnarray} 
leading to the conclusion $X_{n2}=0,$ as before. Continuing in this manner, we conclude $X_{ni}=0$ for $1\leq i \leq n-1.$
%
To complete the proof, we use the same idea as in the concluding part of the proof in Step $2$ of the Proposition \ref{dp}.
\end{proof}
\subsection{\large \sf Rigidity}Finally, we prove a rigidity theorem for the operators in $\mathcal FB_n(\Omega).$ In other words, we show that any intertwining unitary between two operators in the class $\mathcal FB_n(\Omega)$ must be diagonal. We refer to this phenomenon as ``rigidity.''
\begin{thm}\label{hdu}
Any two operators $T$ and $\tilde{T}$ in  $\mathcal{F}B_n(\Omega)$
are unitarily equivalent if and
only if there exists unitary operators $U_i$,  $0\leq i\leq n-1$,
such that $U_iT_i=\tilde{T}_iU_i$ and
$U_iS_{i,j}=\tilde{S}_{i,j}U_j,$  $i<j.$
\end{thm}
\begin{proof}
Clearly, it is enough to prove the necessary part of this statement. Let $U$ be an unitary operator such that
$UT=\tilde{T}U.$
By Proposition \ref{dp}, $U$ must be upper triangular, say $U= \big ( \!\!\big ( U_{ij} \big) \!\!\big )_{i,j=1}^n$ with $U_{ij}=0$ whenever $i >j.$  
Hence for $1\leq i\leq n,$ we have 
$$ U_{ii}T_{i-1i-1}=\tilde{T}_{i-1i-1}U_{ii}.$$ 
Since $U$ is  unitary and upper triangular, it follows that 
$$U_{11}^*U_{11}=I,\;\;\mbox{and}\;\;U^*_{1j}U_{11}=0,\;\; 2\leq j\leq n.$$
However, $U_{11}$  intertwines $T_0$ and $\tilde{T}_0$ and
we have just seen that it is an isometry. It must be then unitary  by Proposition \ref{dr}.   Hence $U_{1j}=0,\; 2\leq j\leq n.$ For any natural number $m<n,$  if we have
$U_{ki}=0,\;1\leq  k\leq m\; ;  k < i\leq n,$
then 
$$U_{m+1}^*U_{m+1}=I\;\;\;\mbox{and}\;\;\;U^*_{m+1i}U_{m+1m+1}=0,\; \;m+1< i\leq n.$$
Since $U_{m+1\,m+1}$ intertwines $T_{m}$ and
$\tilde{T}_{m}$ and it is isometric, we conclude, using Proposition \ref{dr}, that $U_{m+1m+1}$ is  unitary. Hence
$$U_{m+1\,i}=0,\;m+1<i\leq n.$$ An induction on $m$ 
proves that $U$ is  diagonal.
\end{proof}
We use the rigidity theorem just proved to extract a set of  unitary invariants for operators in the class $\mathcal FB_n(\Omega).$

\begin{prop}\label{mit}
Suppose $T$ is an operator in $\mathcal FB_n(\Omega)$ and that  
$t_{n-1}$ is a non-vanishing holomorphic section of $E_{T_{n-1}}.$
Then 
\begin{enumerate}
\item[(i)] the curvature $\mathcal K_{T_{n-1}},$ 
\item[(ii)] $\tfrac{\|t_{i-1}\|}{\|t_i\|},$ where $t_{i-1}=S_{i-1,i}(t_i),\, 1\leq i\leq n-1;$ 
\item[(iii)] $\tfrac{\|S_{k,l}(t_l)\|}{\|t_0\|},\;\, 2\leq l\leq n-1, 0\leq k\leq n-3$
\end{enumerate}
are unitary invariants for the operator $T.$ 
\end{prop}

%
\begin{proof}
Suppose $T,$ $\tilde{T}$ are in $\mathcal FB_n(\Omega)$
and that there is an unitary $U$ such that $TU=\tilde{T}U$. Such an intertwining unitary must be diagonal, that is, 
$U=U_0\oplus \cdots \oplus U_{n-1},$ for some choice of $n$ unitary operators $U_0, \ldots , U_{n-1}.$   
Since $U_iT_i=\tilde{T}_iU_i, 0\leq i\leq n-1,$ and
$U_iS_{i,i+1}=\tilde{S}_{i,i+1}U_{i+1}, 0\leq i\leq n-2,$ we have
\begin{eqnarray}\label{inv1} U_i(t_i(w))=\phi(w)\tilde{t}_i(w),\; 0\leq i\leq
n-1,\end{eqnarray} where $\phi$ is some non zero holomorphic
function. Thus 
$$\mathcal{K}_{T_{n-1}}=\mathcal{K}_{\tilde{T}_{n-1}}\;\;\;\mbox{and}\;\;\; \frac{\|t_{i-1}\|}{\|\tilde{t}_{i-1}\|}=\frac{\|t_i\|}{\|\tilde{t}_i\|},\;1\leq i\leq n-1.$$
For  $2\leq l\leq n-1, 0\leq k\leq n-3$ and $w\in \Omega$, we have
\begin{eqnarray*}
\|S_{k,l}(t_l(w))\|&=& \|U_kS_{k,l}(t_l(w))\|\\
&=& \|\tilde{S}_{k,l}U_l(t_l(w))\|\\
&=& |\phi(w)| \|\tilde{S}_{k,l}(t_l(w))\|\\
&=&\frac{\|t_0(w)\|}{\|\tilde{t}_0(w)\|}{\|\tilde{S}_{k,l}(\tilde{t}_l(w))\|}
\end{eqnarray*}
This completes the proof.
\end{proof}
\begin{rem}
The invariants listed in the preceding theorem are not necessarily complete.  Pick two operators $T$ and $\tilde{T}$ in $\mathcal FB_n(\Omega)$ for which the invariants of Theorem \ref{mit} agree.  Then there exists unitary operators $U_i,$ on the Hilbert space $\mathcal H_i,$ $0\leq i \leq n-1,$ such that
\begin{enumerate}
\item $U_iT_i=\tilde{T}_iU_i,0\leq i\leq n-1$ and
$U_iS_{i,i+1}=\tilde{S}_{i,i+1}U_{i+1},0\leq i\leq n-2;$ 
\item 
$
\|U_kS_{k,l}(x_l)\|={\|\tilde{S}_{k,l}U_l(x_l)\|},\;x_\ell \in \mathcal H_\ell,\;2\leq l\leq
n-1, 0\leq k\leq n-3.$ 
\end{enumerate}
There is no obvious reason why this should be enough for the operators $T$ and $\tilde{T}$ to be unitarily equivalent.
\end{rem}

\begin{prop}
If an operator $T$ is in $\mathcal{F}B_n(\Omega),$ then it is irreducible.
\end{prop}
\begin{proof}

Let $P$ be a projection in the commutant $\{T\}^{\prime}$ of the operator $T.$ The operator $P$ must therefore be upper triangular by  Proposition \ref{utp}. It is also a Hermitian idempotent and therefore must be diagonal with projections  $P_{ii}, 0\leq i\leq n-1,$ on the diagonal.
We are assuming that $PT=TP,$  which gives 
$$P_{ii}S_{i,i+1}=S_{i,i+1}P_{i+1i+1}, \;0\leq i\leq n-2.$$
None of the operators $S_{i,i+1},\;0\leq i\leq n-2,$ are zero by hypothesis. It follows that $P_{ii}=0,$ if and only if $P_{i+1\,i+1}=0.$
Thus, for any projections $P_{ii}\in \{T_i\}^{\prime}$, we
have  only two possibilities:
$$P_{00}=P_{11}=P_{22}=\cdots=P_{n-1n-1}=I,~\mbox{or}~P_{00}=P_{11}=P_{22}=\cdots=P_{n-1n-1}=0.$$ 
Hence $T$ is irreducible.
\end{proof}
\subsection{\large \sf Frames}As in Remark \ref{frame}, we attempt to relate the frame of the holomorphic vector bundle $E_T,$  $T\in \mathcal FB_n(\Omega)$ to that of the direct sum of the line bundles $E_{T_0\oplus \cdots \oplus T_{n-1}}.$  Let $\boldsymbol t=\{{t_0},  {t_1},\ldots, {t_{n-1}}\}$ be a set of non-vanishing holomorphic sections for the line bundles $E_{T_0},\ldots ,E_{T_{n-1}},$ respectively. Suppose that  a suitable linear combination 
of these  non-vanishing sections $t_i,\; i=0,\ldots, n-1,$ and their derivatives produces a holomorphic frame $\boldsymbol \gamma:=\{\gamma_0,\ldots , \gamma_{n-1}\}$ for the vector bundle $E_T,$ that is, 
$$\gamma_{i}=t_{0}^{(i)}+\mu_{1,i}t^{(i-1)}_1+\cdots+\mu_{i-1,i}t^{(1)}_{i-1}
+t_i$$ for some choice of non-zero constants $\mu_{1,i},\ldots, \mu_{i-1,i},\;0\leq i \leq k-1.$ 
The existence of such an orthogonal frame is not guaranteed except when $n=2.$. Assuming that it exists, the relationship between these vector bundles can be very mysterious as shown below.  This justifies, to some extent, the choice of the smaller class of operators in the next section. 
If $\tilde{\boldsymbol t}$ is another set of non-vanishing sections for the line bundles $E_{T_1},\ldots ,E_{T_{n-1}},$ then the linear combination of these  with exactly the same constants $\mu_{ij}$ is a second  holomorphic frame, say $\tilde{\boldsymbol\gamma}$ of the vector bundle $E_T.$ Let $\Phi_k$ be a change of frame between the two sets of non-vanishing orthogonal frames $\boldsymbol t$  and $\tilde{\boldsymbol t},$ and $\Psi_k$ be a change of frame between $\boldsymbol \gamma$ and $\tilde{\boldsymbol\gamma}.$ We now describe the relationship between $\Phi_k$ and $\Psi_k$  explicitly:
\begin{enumerate}
\item $\Phi_k(i,j):=\phi_{i,j}=\psi_{i,j}:=\Psi_k(i,j)=0,$  $i>j,$ that is, $\Phi_{k}$ and
$\Psi_{k}$ are upper-triangular. 
\item For $0\leq i \leq k-1,$ we have $\phi_{i,i}=\psi_{i,i}=\phi_{0,0}$, and for $i<k-1,$ we have 
$$\psi_{i,k-1}=C^i_{k-1}\phi_{0,0}^{(k-1-i)}+\cdots+C^i_{k-1-j}\mu_{j,k-1}\phi_{0,j}^{(k-1-j-i)}+\cdots+\mu_{k-1-i,k-1}\phi_{0,k-1-i},$$
where $C^n_r$ stands for the binomial coefficient ${n\choose r}.$
\item In particular, for $1\leq i \leq k-1,$ if we choose  $\phi_{0,i},$ then 
$\psi_{i,k-1}=C^i_{k-1}\phi_{0,0}^{(k-1-i)}$. In this case, we have  
\begin{enumerate}
\item $$\Psi_{k}=\begin{pmatrix}
\psi & \psi^{(1)} & \psi^{(2)}&\cdots&\psi^{(k-2)}&\psi^{(k-1)}\\
&\psi &2\psi^{(1)}&\cdots&C^{1}_{k-2}\psi^{(k-3)}&C^{1}_{k-1}\psi^{(k-2)} \\
&&\psi &\ddots&\vdots&\vdots \\
&&&\ddots&\ddots&\vdots\\
&& &&\psi&C^{(k-2)}_{k-1}\psi^{(1)}\\
&&&&&\psi\\
\end{pmatrix};$$
\item and there are $\frac{(k-2)(k-1)}{2}$ equations in $\frac{(k-1)k}{2}$ variables, namely, 
$\mu_{i\,j},\; 1\leq i <j,\, j \leq k-1.$ 
Thus these coefficients are determined as soon we make an arbitrary choice of the coefficients  $\mu_{1,k-1},\ldots, \mu_{k-2,k-1}.$ 
\end{enumerate}
\end{enumerate}
We prove the statements (1) and (2) by induction on $k.$
These statements are valid for $k=2$ as was noted in Remark \ref{frame}. To prove their validity for an arbitrary $k\in \mathbb N,$  assume them to be valid for $k-1.$ Let $\Phi^{i}_k$ and $\Psi^{i}_k$ denote the $i$th row of $\Phi$ and $\Psi,$ respectively.  Suppose that $(\widetilde{t}_0,\widetilde{t}_1,\cdots,
\widetilde{t}_{k})=(t_0,t_1,\cdots, t_{k})\Phi_{k}$ and
$(\widetilde{\gamma}_0,\widetilde{\gamma}_1,\cdots,\widetilde{\gamma}_{k})=(\gamma_0,\gamma_1,\cdots,\gamma_{k})\Psi_{k}.$ Then we have
$$\widetilde{t}_j=(t_0,t_1,\cdots, t_{k-1})\Phi^j_{k-1}+t_k\psi_{k,j}, j<k. $$
For any $i<k$, we have 
$$\begin{array}{lll}
\widetilde{\gamma}_i&=&(\gamma_0,\gamma_1,\cdots,
\gamma_{k-1})\Psi^i_{k-1}+\gamma_k\psi_{k,i}\\
&=&(\gamma_0,\gamma_1,\cdots,
\gamma_{k-1})\Psi^i_{k-1}+(t_{0}^{(k)}+\mu_{1,k}t^{(k-1)}_1+\cdots+\mu_{i,k}t^{(k-i)}_{i}+\cdots
+t_k)\psi_{k,i}\\
\end{array}$$ and
$$\widetilde{\gamma}_i=\widetilde{t}_{0}^{(i)}+\mu_{1,i}\widetilde{t}^{(i-1)}_1+\cdots+\mu_{i-1,i}\widetilde{t}^{(1)}_{i-1}
+\widetilde{t}_i,\;i<k.$$ 
From these equations, it follows that
\begin{eqnarray*}
\lefteqn{(\gamma_0,\gamma_1,\cdots,
\gamma_{k-1})\Psi^i_{k-1}+(t_{0}^{(k)}+\mu_{1,k}t^{(k-1)}_1+\cdots+\mu_{i,k}t^{(k-i)}_{i}+\cdots
+t_k)\psi_{k,i}}\\
&=&\widetilde{t}_{0}^{(i)}+\mu_{1,i}\widetilde{t}^{(i-1)}_1+\cdots+\mu_{i-1,i}\widetilde{t}^{(1)}_{i-1}
+\widetilde{t}_i. \phantom{GADADHARMISRAGADADHAR}\end{eqnarray*}
We Note that $\mu_{i,k}\psi_{k,i}t^{(k-i)}_i$ appears only once in this equation to conclude  $\psi_{k,i}=0,$ $i<k.$  Comparing the coefficients of $t_i$ on both sides of the
equation, we  also conclude that $\psi_{k,i}=\phi_{k,i}, i<k$ completing the induction step for the first statement of our claim.

Our assumption that  $(\widetilde{t}_0,\widetilde{t}_1,\cdots,
\widetilde{t}_{k})=(t_0,t_1,\cdots, t_{k})\Phi_{k}$ and
$(\widetilde{\gamma}_0,\widetilde{\gamma}_1,\cdots,\widetilde{\gamma}_{k})=(\gamma_0,\gamma_1,\cdots,\gamma_{k})\Psi_{k}$ gives 
$$\sum\limits_{i=0}^k(t^{i}_0+\mu_{1,i}t^{(i-1)}_1+\cdots+\mu_{i-1,i}t^{(1)}_{i-1}+t_i)\psi_{i,k}=\sum\limits_{i=0}^{k}\mu_{i,k}(t_0\phi_{0,i}+\cdots+t_i\phi_{0,0})^{(k-i)}, i<k. $$
A comparison of the coefficients of $t^{(i)}_0$ leads to
$$\psi_{i,k}=C^i_{k}\phi_{0,0}^{(k-i)}+\cdots+C^{i}_{k-j}\mu_{j,k}\phi_{0,j}^{(k-j-i)}+\cdots+\mu_{k-i,k}\phi_{0,k-i}, i<k$$
completing the proof of the second statement. For the third statement, from the equations
\begin{eqnarray*}\lefteqn{
\sum\limits_{i=0}^{k-1}(t^{i}_0+\mu_{1,i}t^{(i-1)}_1+\cdots+\mu_{i-1,i}t^{(1)}_{i-1}+t_i)\psi_{i,k-1}}\\&&=\sum\limits_{i=0}^{k-1}\mu_{i,k-1}(t_0\phi_{0,i}+\cdots+t_i\phi_{0,0})^{(k-1-i)}, i<k-1,
\end{eqnarray*}
setting  $\phi_{0,i}=0,$  and comparing the coefficients of
$t_i$, $i>0,$ we have that $\phi_{i,k-1}=c_{i,k-1}\phi^{(k-1-i)}_{0,0}$ for some $c_{i,k-1}\in \mathbb{C}.$ Putting this back in  the equation given above, we obtain $\frac{(k-2)(k-1)}{2}$ equations involving $\frac{(k-1)k}{2}$ coefficients. This 
completes the proof of the third statement. 
\subsection{\large \sf An even smaller class}The relationship between the non-vanishing holomorphic sections of the vector bundles $E_0, \ldots , E_{n-1}$ and the holomorphic frame of the vector bundle $E$ of rank $n$ is rather complex, in general, as we have just  seen. The theorem below shows that it is simple provided we impose additional restrictions. 
\begin{prop}
For an operator $T$ in the Cowen-Douglas class ${B}_n(\Omega),$ acting on a complex separable Hilbert space $\mathcal H,$ the following conditions are equivalent.
\begin{enumerate}
\item There exists an orthogonal decomposition
$\mathcal{H}_0\oplus\mathcal{H}_1\oplus \cdots \oplus
\mathcal{H}_{n-1}$ of the Hilbert space $\mathcal{H}$ operators
$T_k:\mathcal{H}_k\to\mathcal{H}_k$ in $B_1(\Omega),$ $k=0,1,\cdots, n-1;$ 
$S_{k-1,k}:\mathcal{H}_{k}\to\mathcal{H}_{k-1},$  
$k=1,2,\cdots,n-1,$ such that 
$$T=\left(\begin{array}{ccccccccccc}T_{0}&S_{0,1}&0&\cdots&0\\
0&T_{1}&S_{1,2}&\ddots&0\\
\vdots&\vdots&\ddots&\ddots&\vdots\\
0&0&0&T_{n-2}&S_{n-2,n-1}\\
0&0&0&0&T_{n-1} \end{array}\right)$$
and $T_{k-1}S_{k-1,k}=S_{k-1,k}T_{k},\;0\leq k\leq n-1.$
\item There exists a holomorphic  frame
$\{\gamma_0,\gamma_1,\cdots,\gamma_{n-1}\}$ of the vector bundle  $E_{T}$ such that $t_k(w)$ is orthogonal to $t_j(w),$ $w\in \Omega,$ whenever $k\neq j,$  and 
$$t_k(w):=\sum_{j=0}^{k}\tfrac{1}{j!}\tfrac{\partial^j}{\partial
w^j}\gamma_{k-j}(w),\; 0\leq k\leq
n-1.$$
\end{enumerate}
\end{prop}
\begin{proof}
We prove $(1)$ implies $(2).$  Let $t_{n-1}$ be a holomorphic frame of the line bundle $E_{T_{n-1}}$. Set $t_{i}=S_{i,i+1}(t_{i+1}),$ $0\leq i\leq
n-2.$ By shrinking $\Omega$ to a smaller open set, we may assume that $t_i$ is a non-vanishing holomorphic section of the line bundle $E_{T_i}$, $0\leq i\leq n-1.$ Define
$\gamma_k$  recursively from the equations 
$$ \gamma_0(w):=t_0(w)\;\;\;\mbox{and}\;\;\gamma_k(w):=t_k(w)-\sum_{j=1}^k\tfrac{1}{j!}\tfrac{\partial^j}{\partial
w^j}\gamma_{k-j}(w),\;1\leq k\leq n-1.$$ For $w$ in $\Omega,$ it is easy to see that the set of vectors  $\{\gamma_k(w):0\leq k\leq n-1\}$ is linearly independent and we have
\begin{eqnarray*}
(T-w)\gamma_i(w)&=&(T-w)\big(t_i(w)-\sum_{j=1}^i\tfrac{1}{j!}\tfrac{\partial^j}{\partial
w^j}\gamma_{i-j}(w)\big)\\
&=&t_{i-1}(w)-\sum_{j=1}^i\tfrac{1}{(j-1)!}\tfrac{\partial^{j-1}}{\partial
w^{j-1}}\gamma_{i-j}(w)\\
&=&t_{i-1}(w)-\sum_{l=0}^{i-1}\tfrac{1}{l!}\tfrac{\partial^{l}}{\partial
w^{l}}\gamma_{i-1-l}(w)\\
&=& t_{i-1}(w)-t_{i-1}(w)\\
&=&0.
\end{eqnarray*}
Hence $\{\gamma_k: 0\leq k\leq n-1\}$ is a holomorphic frame of
$E_T$.

To show that $(2)$ implies $(1),$ we follow the proof given for proving the implication ``(iii) implies (i)'' in Proposition \ref{f}.
\end{proof}
The equivalent conditions of the preceding theorem naturally lead to the definition of a somewhat smaller class of operators given below. 
\begin{defn}
Given a set of operators $T_0, \ldots, T_{n-1}$ in the Cowen Douglas class $B_1(\Omega)$ and a set of non-zero operators $S_{i\,i+1}$ obeying the commutation relation  $T_iS_{i,i+1}=S_{i,i+1}T_{i+1},$ $0\leq i \leq n-2,$ we let $\widetilde{\mathcal{F}}B_n(\Omega)$ denote the set of 
operators $T$  of the form $$ T=\begin{pmatrix}
T_{0} & S_{0,1} & 0 &\cdots &0&0\\
0 &T_{1}&S_{1,2}&\cdots &0&0 \\
\vdots&\ddots &\ddots&\ddots& \vdots&\vdots\\
0&\cdots & 0 & T_{n-3} & S_{n-3,n-2} & 0\\
0&\cdots&0& 0&T_{n-2}& S_{n-2,n-1}\\
0&\cdots&\cdots&0& 0 &T_{n-1}\\
\end{pmatrix}.$$ 
\end{defn}
Any holomorphic change of frame of the vector bundle $E_T$ for $T$ in $\tilde{\mathcal F}B_n(\Omega)$ must be of the form $\big (\!\!\big ({j \choose i} \phi^{(j-i)} \big )\!\!\big),$ where the binomial coefficients ${j\choose i}$ are set to $0$ if $i > j,$ and corresponds to a change of frame for the vector bundle $E_{T_0\oplus \cdots \oplus T_{n-1}}$ of the form $\phi\oplus \cdots \oplus \phi$ just as in  Remark \ref{frame} for the case of rank $2.$  

The following proposition shows that the operators in this smaller class are not only irreducible but are often strongly irreducible.
We have given two separate sufficient conditions.  The second of these conditions was also necessary for strong irreducibility in rank $2,$ where the two classes $\mathcal FB_2(\Omega)$ and $\tilde{\mathcal F}B_2(\Omega)$ coincide. But we haven't been able to determine if this condition is also necessary in general. 
\begin{prop} Suppose $T$ is an operator in $\widetilde{\mathcal{F}}B_n(\Omega).$ Then the operator $T$ is strongly irreducible if the operators $S_{i,i+1}$
\begin{enumerate}
\item  are  invertible, or
\item  are not in $\mbox{ran}\,\sigma_{T_{i},T_{i+1}},$ 
\end{enumerate}
$0\leq i \leq n-2.$
\end{prop}
\begin{proof}
Suppose that the operators $S_{i,i+1},
0\leq i \leq n-2,$ are  invertible. For $1\leq k \leq n-1,$ let $X_k$ be the block diagonal operator with $I, \ldots I, S_{k,k+1},  \ldots ,S_{n-2,n-1}$ on its diagonal in that order. A straightforward computation shows that 
%
%
%
%
%
%
%
%
$$X_nX_{n-1}\cdots X_1TX^{-1}_nX^{-1}_{n-1}\cdots
X^{-1}_1=\begin{pmatrix}
T_{0} & I & \\
&T_{0}&I& \\
&&\ddots&\ddots&\\
&&&T_{0}&I\\
&&&&T_{0}.\\
\end{pmatrix}
$$
Since the operator $T_{0}$ is assumed to be in $B_1(\Omega),$  therefore it is strongly irreducible (cf. \cite[Proposition 2.28]{jw}). 

Let $J_n[T_0]$ denote the operator appearing on the right hand side of the equation displayed above. We claim that $J_n[T_0]$ is strongly irreducible. We have shown in Proposition \ref{pirre} that  the operator $J_2[T_0]$  is strongly irreducible.
Assume that $J_k[T_0]$ is strongly irreducible for all $k<m.$  Now, any idempotent $P$ in the commutant of $J_m[T_0]$ must look like 
$\Big (\begin{smallmatrix} I_{m-1\times m-1} & P_{m-1\times 1}\\ 0& P_{m\,m}.\end{smallmatrix} \Big ),$ or $\begin{pmatrix}0_{m-1\times m-1}& P_{m-1\times 1}\\
0 & P_{mm}\end{pmatrix}$. Thus $P_{mm}$ must commute with $T_0$
and hence $P_{mm}=0\;\mbox{or}\;I$.  By using
Theorem 2.19, it is easy to see that  if $P= \begin{pmatrix}I_{m-1\times m-1}& P_{m-1\times 1}\\
0 & 0 \end{pmatrix}$ or  $P= \begin{pmatrix}0_{m-1\times m-1}& P_{m-1\times 1}\\
0 & I\end{pmatrix},$ then $P$ does not belong to commutant of
$J_m[T_0].$  Thus $P= \begin{pmatrix}I_{m-1\times m-1}& P_{m-1\times 1}\\
0 & I\end{pmatrix}$ or  $P=\begin{pmatrix}0_{m-1\times m-1}& P_{m-1\times 1}\\
0 & 0\end{pmatrix}$. Now, using the equation $P^2=P$, we conclude
that the $m-1\times 1$ block $P_{m-1\times 1}$ in $P$ must be
zero. Thus $P=I$ or $P=0$. This inductively proves that $J_n[T_0]$
must be strongly irreducible.

The operator $T$ is similar to a strongly irreducible operator and consequently it is strongly irreducible as well. This completes the proof of $(1)$.

Let $P$ be an idempotent in the commutant $\{T\}^{\prime}$ of the operator $T.$ The operator $P$ is then upper triangular, say, $P= \big ( \!\!\big ( P_{ij} \big) \!\!\big )_{i,j=1}^n$ with $P_{ij}=0$ whenever $i >j$ by Proposition \ref{utp}. 
Since $P$ is in $\{T\}^\prime$ and it is upper triangular, a  straightforward matrix computation shows that $P_{ii}$ must  be in $\{T_{n-i}\}^\prime,$ $1\leq i \leq n.$ Thus the only possible choice for the operator $P_{ii}$ is 
$I$ or $0$. Suppose it were possible to choose $P_{11}=I$ and
$P_{22}=0.$ Then we have that
$$S_{0,1}+P_{12}T_{1}=T_{0}P_{12}$$
contradicting the assumption that $S_{0,1}\not\in
\mbox{ran}\,\sigma_{T_{0}, T_{1}}$. Similarly, for any $i\leq n-1$,
if $P_{ii}=I$ and $P_{i+1i+1}=0,$ then 
$S_{i,i+1}$ must belong to $\mbox{ran}\,\sigma_{T_{i},T_{i+1}},$ again contradicting our hypothesis. We therefore conclude that the only possible idempotents in $\{T\}^{\prime}$ are
either $0$ or $I$. Hence  $T$ is strongly irreducible. This completes the proof of $(2).$

\end{proof}
\subsection{\large \sf A complete set of unitary invariants for operators in $\tilde{\mathcal F}B_n(\Omega)$} It is easy to give a set unitary invariants for this class which is complete.  
\begin{thm}\label{uinvs}
Suppose $T$ is an operator in $\tilde{\mathcal F}B_n(\Omega)$ and that  
$t_{n-1}$ is a non-vanishing holomorphic section of $E_{T_{n-1}}.$
Then 
\begin{enumerate}
\item[(i)] the curvature $\mathcal K_{T_{n-1}},$ 
\item[(ii)] $\tfrac{\|t_{i-1}\|}{\|t_i\|},$ where $t_{i-1}=S_{i-1,i}(t_i),\, 1\leq i\leq n-1$ 
\end{enumerate}
are a complete set of unitary invariants for the operator $T.$
%
\end{thm}
\begin{proof}
Pick any two  operators $T$ and $\tilde{T}$ in $\tilde{\mathcal F}B_n(\Omega)$ and assume that they are unitarily equivalent via some unitary operator $U,$ that is, $UT=\tilde{T}U.$ By Theorem
\ref{hdu}, there exist $U_i:\mathcal{H}_i\to
\widetilde{\mathcal{H}}_i$ such that $U=\oplus_{i=0}^{n-1} U_i.$
Thus $U_iT_i=\tilde{T}_iU_i,\, 0\leq i\leq n-1,$ and
$U_iS_{i,i+1}=\tilde{S}_{i,i+1}U_{i+1},\,0\leq i\leq n-2.$ Consequently, 
\begin{eqnarray}\label{inv2} U_i(t_i(w))=\phi(w)\tilde{t}_i(w),\; 0\leq i\leq n-1,\end{eqnarray} 
where $\phi$ is some non zero holomorphic
function. By equation (\ref{inv2}), we have
$$\mathcal{K}_{T_{n-1}}=\mathcal{K}_{\tilde{T}_{n-1}}\;\;\mbox{and}\;\; \frac{\|t_{i-1}\|}{\|\tilde{t}_{i-1}\|}=\frac{\|t_i\|}{\|\tilde{t}_i\|},
\;\;1\leq i\leq n-1.$$

Conversely if we assume that $T$ and $\tilde{T}$  are operators in $\tilde{\mathcal F}B_n(\Omega)$ for which these invariants are the same, then there exist a non-zero holomorphic function $\phi$ defined on $\Omega$ such that
$$\|t_i(w)\|=|\phi(w)|\,\|\tilde{t}_i(w)\|\; 0\leq i\leq n-1.$$
For $ 0\leq i\leq n-1,$ define an operator $U_i:\mathcal{H}_i\to \widetilde{\mathcal{H}}_i$ as
follows: 
$$U_i(t_i(w))=\phi(w)\tilde{t}_i(w),\; w\in
\Omega.$$ 
For $0\leq i\leq n-1,$
\begin{eqnarray*}
\|U_i(t_i(w))\|&=&\|\phi(w)\tilde{t}_i(w)\|\\
&=& |\phi(w)|\|\tilde{t}_i(w)\|\\
&=&\|t_i(w)\|.
\end{eqnarray*}
Thus $U_i$ extend to an isometry from $\mathcal{H}_i$ to
$\widetilde{\mathcal{H}}_i$ and $U_i T_i=\tilde{T}_i U_i$. Since
$U_i$ is isometric and $U_i T_i=\tilde{T}_i U_i$, it follows, using  Proposition \ref{dr}, that $U_i$ is unitary. It is easy to see that $
U_iS_{i,i+1}=\tilde{S}_{i,i+1}U_{i+1}$ for $0\leq i\leq n-2$ also. Hence setting $U= U_0 \oplus \cdots \oplus U_{n-1},$ we see that 
$U$ is unitary and $UT=\tilde{T}U$ completing the proof.
\end{proof}

\section{An application to Module tensor products}
The localization of a module at a point of the spectrum is obtained by tensoring with the one dimensional module
of evaluation at that point. The  localization technique has played a prominent role in the structure theory of modules. More recently, they have found their way into the study of Hilbert modules (cf. \cite{dp}). An initial attempt was made in \cite{dmc} to see if higher order localizations would be of some use in obtaining invariants for quotient Hilbert modules. Here we give an explicit description of the module tensor products over the polynomial ring in one variable.  

There are several different ways in which one may define the action of the polynomial ring on $\mathbb C^k.$ The following lemma singles out the possibilities for the module action which evaluates a function at $w$ along with a finite number of its derivatives, say $k-1$, at $w.$  Let $f$ be a polynomial in one variable. Set  
$$\mathcal{J}_{\boldsymbol \mu}(f)(z)=\left(
                         \begin{array}{cccc}
                           f(z) & 0 & \cdots & 0 \\
                           \mu_{2,1}\tfrac{\partial }{\partial z}f(z) & f(z) & \cdots & 0 \\
                           \vdots & \vdots & \ddots & \vdots \\
                           \mu_{k,1}\tfrac{\partial^{k-1}}{\partial z^{k-1} }f(z) & \mu_{k-1,1}\tfrac{\partial^{k-2}}{\partial z^{k-2}}f (z)& \cdots & f(z) \\
                         \end{array}
                       \right)
.$$
\begin{lem}
${\mathcal{J}_{\boldsymbol \mu}}(fg)=\mathcal J_{\boldsymbol \mu}(f)\mathcal J_{\boldsymbol \mu}(g)$
if and only if
$$(p+1-j-l)\mu_{p+1-j,l}=\mu_{p+1-j,l+1}\,\mu_{l+1,l},\;1\leq l\leq p-2,\;1\leq j< p-l+1,\mbox{and}\;\mu_{i,i}=1,\;1\leq i\leq k,$$
if and only if $$\mu_{p,l}\;\mu_{l,i}=\tbinom
{p-i}{l-i}\mu_{p,i},\;1\leq p,l,i\leq k,\,i\leq l\leq
p\;\mbox{and}\;\mu_{i,i}=1,\;1\leq i\leq k.$$
\end{lem}
\begin{proof} The necessity part of the proof is a straightforward verification. To prove the sufficiency, for $1\leq i,j\leq k$ and $i\leq j,$ note that 
\begin{eqnarray*}
(\mathcal J_{\boldsymbol \mu}(f)(z)\mathcal J_{\boldsymbol \mu}(g)(z))_{i,j}&=&
\sum_{l=0}^{i-j}\mu_{i,j+l}\,\mu_{j+l,j}
(\tfrac{\partial^{i-j-l}}{\partial z^{i-j-l} }f(z)) (\tfrac{\partial^l}{\partial z^l}g(z))\\
&=&\sum_{l=0}^{i-j} \tbinom {i-j}{i-j-l}\mu_{i,j} (\tfrac{\partial^{i-j-l}}{\partial z^{i-j-l} }f(z)) (\tfrac{\partial^l}{\partial z^l}g(z))\\
&=&\mu_{i,j}\sum_{l=0}^{i-j} \tbinom {i-j}{i-j-l}(\tfrac{\partial^{i-j-l}}{\partial z^{i-j-l} }f(z)) (\tfrac{\partial^l}{\partial z^l}g(z))\\
&=& \mu_{i,j} \tfrac{\partial^{i-j}}{\partial z^{i-j}}(f g)(z)\\
&=& (\mathcal J_{\boldsymbol \mu}(f g)(z))_{i,j}.
\end{eqnarray*}
For $i>j,$ 
$$(\mathcal J_{\boldsymbol \mu}(f)(z)\mathcal J_{\boldsymbol \mu}(g)(z))_{i,j}=(\mathcal J_{\boldsymbol \mu}(f
g)(z))_{i,j}=0.$$ Hence we have
$$\mathcal J_{\boldsymbol \mu}(fg)=\mathcal J_{\boldsymbol \mu}(f)\mathcal J_{\boldsymbol \mu}(g).$$
\end{proof}
For $\mathbf x$ in $\mathbb{C}^k,$ and $f$ in the polynomial ring $P[z],$ define the module action as follows: $$f\cdot \mathbf x=\mathcal J_{\boldsymbol \mu}(f)(w)\mathbf x.$$ 

Suppose $T_0:\mathcal M\to \mathcal M$ is an operator in ${B}_1(\Omega).$ Assume that the operator $T$ has been realized as the adjoint of a multiplication operator acting on a Hilbert space of functions possessing a reproducing kernel $K$. Then the polynomial ring acts on the Hilbert space  $\mathcal{M}$ naturally by point-wise multiplication making it a module. We construct a module of $k$ - jets by setting 
$$J\mathcal{M}=\Big \{\sum_{l=0}^{k-1}\tfrac{\partial^i}{\partial
z^i}{h}\otimes {\epsilon}_{i+1}: h\in\mathcal{M}\Big \},$$ 
where $\epsilon_{i+1}, \, 0 \leq i \leq k-1,$ are the standard basis vectors in $\mathbb C^k.$ There is a natural module action on $J\mathcal M,$ namely, 
$$\Big (f, \sum_{l=0}^{k-1}\tfrac{\partial^i}{\partial z^i}{h} \Big )\mapsto
\mathcal{J}(f)\Big (\sum_{l=0}^{k-1}\tfrac{\partial^i}{\partial
z^i}{h}\otimes {\epsilon}_{i+1}\Big ),\, f\in P[z],\,
h\in\mathcal{M},$$ 
where 
\beqa \mathcal{J}(f)_{i,j} =  \begin{cases}  {{i-1}\choose{j-1}}
\partial^{i-j}f &\mbox{if} \,\, i\geq j,\\
0 & \mbox{otherwise}.
\end{cases}
\eeqa
The module tensor product $J\mathcal{M}\otimes_{\mathcal{A}({\Omega})}\mathbb{C}_w^k$ is easily identified with the quotient module $\mathcal{N}^{\bot},$ where $\mathcal N\subseteq \mathcal M$ is the sub-module spanned by the vectors  
$$ \big\{ \sum_{l=1}^{k}(J_f\cdot{\bf
h}_l\otimes \epsilon_l- {\bf h}_l\otimes
({\mathcal{J}_{\boldsymbol \mu}}(f))(w)\cdot\epsilon_l): {\bf h}_l\in
J\mathcal{M},\epsilon_l
\in \mathbb{C}^k, f\in P[z] \big\}.$$
Following the  proof of the lemma  \ref{lem1}  in \cite[Lemma
4.1]{dmc}, we can prove:
\begin{lem}\label{lem1}
The module tensor product
$J\mathcal{M}\otimes_{P[z]}\mathbb{C}_w^k$ is
spanned by the vector $e_p(w)$ in
$J\mathcal{M}\otimes_{\mathcal{A}({\Omega})}\mathbb{C}_w^k$, where
$$e_p(w)=\sum_{l=1}^p
b_{p,l}JK(\cdot,w)\epsilon_{p-l+1}\otimes\epsilon_l,\;1\leq p\leq
k$$ where
$$b_{p,l}= \frac{\mu_{p-j+1,l}}{\tbinom {p-l} {j-1}} b_{p,p-j+1},\;l+j< p+1.$$
\end{lem}
The set of vectors $\{e_p(w):w\in\Omega^*,\,1\leq p\leq k\}$ define a natural holomorphic frame for a vector bundle, say $J_{\mbox{\scriptsize loc}}(\mathcal{E}).$ This vector bundle also inherits a Hermitian structure from that of $J\mathcal{M}\otimes_{\mathcal{A}({\Omega})}\mathbb{C}_w^k,$ which furthermore defines a  positive definite kernel on $\Omega\times \Omega:$
\begin{eqnarray*}
J_{\mbox{\scriptsize loc}}K(z,w)&=& \big (\!\!\big (\langle e_p(w),e_q(z) \rangle\big )\!\!\big )\\
&=& \sum_{l=1}^{k} D(l)J_{k-l+1} K(z,w)D(l),
\end{eqnarray*}
where $J_r K(z,w)= \begin{pmatrix} 0_{k-r\times k-r} & 0_{k-r\times r}\\
 0_{r\times k-r} & \tilde{J}_r K(z,w)\\
 \end{pmatrix}$ and $D(l)$ is diagonal. 
Moreover, $D(l)_{m,m}=b_{m+l-1,l}$ and

$$ \tilde{J}_r K(z,w)=\begin{pmatrix} K(z,w) &  \tfrac{\partial}{\partial \bar w} K(z,w) & \cdots & \tfrac{\partial^{r-1}}{\partial\bar{w}^{r-1}}K(z,w) \\
\frac{\partial}{\partial z} K(z,w) & \tfrac{\partial^2}{\partial z  {\partial} \bar w} K(z,w)& \cdots & \tfrac{\partial^r}{\partial z  {\partial} {\bar w}^{r-1}} K(z,w) \\
\vdots & \vdots & \ddots & \vdots\\
\tfrac{\partial^{r-1}}{\partial z^{r-1}} K(z,w) & \frac{\partial^r}{\partial z^{r-1} \partial \bar {w}}K(z,w)& \cdots & \tfrac{\partial^{2r-2}}{\partial z^{r-1}\partial{\bar w}^{r-1}} K(z,w)\\
\end{pmatrix}.$$
 
The two Hilbert spaces $\mathcal{M}$ and $\mathcal{M}\otimes \mathbb{C}^{k}$ may be identified via the map $J_{k-l+1},$ which is given by the formula 
 $$J_{k-l+1}(h)= \sum_{p=0}^{k-l} b_{p+l-1,l}\tfrac{\partial^p}{\partial z^p} h \otimes \epsilon_{p+l}.$$ Since  
 $J_{k-l+1}$ is injective, we may choose an inner product on $J_{p-l+1}{\mathcal{M}}$  making it  unitary.

\begin{prop}\cite[Proposition 4.2]{dmc} The Hilbert module $J_{\mbox{\scriptsize loc}}(\mathcal{M})$ admits a direct sum decomposition of the form $\oplus_{l=1}^{k}J_{k-l+1}{\mathcal{M}},$ and the corresponding reproducing kernel is the sum 
$$\sum_{l=1}^{k}D(l)J_{k-l+1} K(z,w)D(l).$$ 
\end{prop}
Let $\gamma_0$ be a non-vanishing holomorphic section for the line bundle $E$ corresponding to the operator $T_0.$ Put $b_{1,1}t_0(w)=\gamma_0(w)$  and for $1\leq l\leq k-1,$ let \begin{enumerate}
\item $t_l(w):=\sum_{i=0}^{k-l-1}\overline{b}_{l+1+i,l+1}\tfrac{\partial^{i}}{\partial z^i}K(\cdot,w)\otimes\epsilon_{l+1+i},$ 
\item  $\gamma_l(w)=\sum_{i=1}^{l+1}b_{l+1,i}\tfrac{\partial^{l+1-i}}{{\partial}\bar{w}^{l+1-i}}t_{i-1}(w).$
\end{enumerate}
%
%
%
Now, $\{\gamma_0,\gamma_1,\cdots,\gamma_{k-1}\}$ are eigenvectors of the operator $M_z^*-\bar w$ acting on the Hilbert space $\mathcal M_{\rm loc}$.

Since $(M_z^*-\bar w)\gamma_1(w)=0,$ it follows that $(M_z^*-\bar w)t_1(w)=-\frac{b_{2,1}}{b_{2,2}}t_0(w),$ which is equivalent to $(M_z^*-\bar w)t_1(w)=-\mu_{2,1}t_0(w).$ 

Suppose $(M_z^*-\bar w)t_l(w)=-\mu_{l+1,l}t_{l-1}(w)$ 
 for $1\leq l\leq r.$ 
Again, since $(M_z^*-\bar w)\gamma_{r+1}(w)=0,$ it follows that 
 \begin{eqnarray*}
\lefteqn{(M_z^*-\bar w)t_{r+1}(w)}\\
 &=& \tfrac{1}{b_{r+2,r+2}}
 \big\{(-(r+1) b_{r+2,1}{\bar\partial}^{r}t_0(w))\\
 &&
 -\sum_{i=2}^{r+1}b_{r+2,i}(-\mu_{i,i-1}
 {\bar \partial}^{r+2-i}t_{i-2}(w)+(r+2-i){\bar\partial}^{r+1-i}t_{i-1}(w))\big\}\\
 &=&\tfrac{1}{b_{r+2,r+2}}\big\{ \sum_{i=1}^{r}(-(r+2-i)b_{r+2,i}+b_{r+2,i+1}\mu_{i+1,i}){\bar\partial}^{r+1-i}t_{i-1}(w)
 - b_{r+2,r+1}t_{r}(w)\big\}\\
 &=& \tfrac{b_{r+2,r+1}}{b_{r+2,r+2}}t_{r}(w)\\
 &=& \mu_{r+2,r+1}t_{r}(w)
 \end{eqnarray*}

Let  $\Gamma:=J_k\oplus J_{k-1}\oplus\ldots\oplus J_1,$ be the unitary from $\widetilde{\mathcal M}:=\mathcal{M}_0\oplus\cdots \mathcal{M}_{k-1}$ to $\mathcal M_{\rm loc},$ where each of the summands $\mathcal{M}_0,\ldots , \mathcal M_{k-1}$ is equal to $\mathcal{M}.$ Let 
 $K_l(\cdot,w):=J_{k-l}^*t_l(w)=K(\cdot,w),$ $0\leq l\leq k-1.$
Now, we describe the operator $T:= \Gamma^* M^* \Gamma,$ where $M$ is the multiplication operator on $\mathcal M_{\rm loc}.$ 
For $1\leq l\leq k-1,$ set $T_l:= P_{\mathcal{M}_l}T_{|\mathcal{M}_l}$ and note that  
%
%
%
\begin{eqnarray*}
T(K_l(\cdot,w))&=& (\Gamma^* M^*\Gamma) K_l(\cdot,w)\\
&=&\Gamma^* M_z^* t_l(w)\\
&=&\Gamma^*(\bar w t_l(w)+\mu_{l+1,l}t_{l-1}(w))\\
&=&\bar w K_l(\cdot,w)+ \mu_{l+1,l}K_{l-1}(\cdot,w).
\end{eqnarray*}
Now,
\begin{eqnarray*}
T_{l}(K_l(\cdot,w))&=&P_{\mathcal{M}_l}T_{|\mathcal{M}_l}(K_l(\cdot,w))\\
&=& P_{\mathcal{M}_l}T (K_l(\cdot,w))\\
&=& P_{\mathcal{M}_l}(\bar w K_l(\cdot,w)+ \mu_{l+1,l}K_{l-1}(\cdot,w))\\
&=& \bar w K_l(\cdot,w).
\end{eqnarray*}
Let $S_{l-1,l}:\mathcal{M}_l\to \mathcal{M}_{l-1}$ be the bounded
linear operator defined by the rule $S_{l-1,l}(K_l(\cdot,w)):=
\mu_{l+1,l}K_{l-1}(\cdot,w),\,1 \leq l\leq k-1 $. Since
$\mathcal{M}_l=\mathcal{M}_{l-1}=\mathcal{M}$, it follows that
$S_{l-1,l}=\mu_{l+1,l}I$. Hence operator $T$ has the form:
  $$T=\begin{pmatrix} T_0 & \mu_{2,1} I & 0 &\cdots & 0 & 0 \\
  0 & T_0 & \mu_{3,2}I & \cdots & 0  & 0\\
  0 & 0 & T_0 & \ddots & \vdots &\vdots\\
  \vdots& \vdots & \vdots &  \ddots&  \mu_{k-1,k-2}I & 0 \\
  0 & 0 & 0 & \cdots & T_{0} & \mu_{k,k-1}I\\
  0 & 0 & 0 & \cdots & 0 & T_{0}\\
  \end{pmatrix}.$$
Thus $T$ is in $\tilde{\mathcal F}B_k(\Omega)$ and defines, up to unitary equivalence via the unitary $\Gamma,$  the module action in $\mathcal M_{\rm loc}.$ In consequence, setting $\mathbb C^k_w[\boldsymbol \mu]$ to be the Hilbert module with the module action induced by $\mathcal J_{\boldsymbol \mu}(f)(w),$ we have the following theorem as a direct application of Theorem \ref{uinvs}. 
\begin{thm}
The Hilbert modules corresponding to the localizations $J\mathcal M\otimes_{P[z]} \mathbb C^k_w[\boldsymbol \mu_i],$ $i=1,2,$ are in $B_k(\Omega)$ and they are isomorphic if and only if $\boldsymbol \mu_1 = \boldsymbol \mu_2.$
\end{thm}

\bibliographystyle{amsplain}
\bibliography{bibliography}
\end{document}